\documentclass{article}

\usepackage[T1]{fontenc}
\usepackage[utf8]{inputenc}

\usepackage{lmodern}

\usepackage{bm}

\usepackage{color}

\usepackage{microtype}

\usepackage{amsmath}

\usepackage{amsthm}

\usepackage{amssymb}

\usepackage[showonlyrefs]{mathtools}

\usepackage{array}

\usepackage{authblk}

\usepackage[letterpaper,top=2in, bottom=1.5in, left=1in, right=1in]{geometry}

\usepackage{fancyhdr, blindtext}
\newcommand{\changefont}{\fontsize{7}{9}\selectfont}
\newcolumntype{L}[1]{>{\raggedright\let\newline\\\arraybackslash\hspace{0pt}}m{#1}}
\newcolumntype{C}[1]{>{\centering\let\newline\\\arraybackslash\hspace{0pt}}m{#1}}
\newcolumntype{R}[1]{>{\raggedleft\let\newline\\\arraybackslash\hspace{0pt}}m{#1}}

\usepackage[colorinlistoftodos]{todonotes}

\usepackage[backend=biber,style=ieee,doi=false,isbn=false,url=false]{biblatex}

\usepackage[colorlinks=true, linkcolor=blue, citecolor=blue, urlcolor=blue, pdfborder={0 0 0},bookmarksnumbered,bookmarksdepth=3]{hyperref}

\usepackage{mathtools}
\usepackage[normalem]{ulem} 

\usepackage{amsmath}
\usepackage{amssymb}
\usepackage{mathtools}
\usepackage{mathrsfs}

\newcommand{\cut}[1]{{}}

\newcommand{\EE}{{\mathbb{E}}} 				
\newcommand{\PP}{\mathbb{P}} 				
\newcommand{\RR}{\mathbb{R}}				
\newcommand{\HH}{\mathbb{H}}				
\newcommand{\NN}{\mathbb{N}}				

\DeclareMathOperator*{\argmin}{arg\,min}

\newcommand{\DeclareAutoPairedDelimiter}[3]{%
	\expandafter\DeclarePairedDelimiter\csname Auto\string#1\endcsname{#2}{#3}%
	\begingroup\edef\x{\endgroup
		\noexpand\DeclareRobustCommand{\noexpand#1}{%
			\expandafter\noexpand\csname Auto\string#1\endcsname*}}%
	\x}

\DeclareAutoPairedDelimiter{\p}{(}{)} 					
\DeclareAutoPairedDelimiter{\sp}{[}{]} 					
\DeclareAutoPairedDelimiter{\abs}{|}{|} 					
\DeclareAutoPairedDelimiter{\cp}{\{}{\}} 				
\DeclareAutoPairedDelimiter{\dotp}{\langle}{\rangle} 	
\DeclareAutoPairedDelimiter{\n}{\Vert}{\Vert} 			
\DeclareAutoPairedDelimiter{\cl}{\lceil}{\rceil}

\newcommand{\cF}{{\mathcal{F}}}
\newcommand{\cG}{{\mathcal{G}}}

\newcommand{\cO}{{\mathcal{O}}}

\usepackage{amsmath}
\usepackage{amsthm}
\usepackage{amssymb}
\usepackage{mathtools}
\usepackage{mathrsfs}

\usepackage[nameinlink]{cleveref}

\newcommand{\bc}{\begin{center}}
\newcommand{\ec}{\end{center}}

\newcommand{\bdm}{\begin{displaymath}}
\newcommand{\edm}{\end{displaymath}}

\newcommand{\beq}{\begin{equation}}
\newcommand{\eeq}{\end{equation}}

\newcommand{\bfl}{\begin{flushleft}}
\newcommand{\efl}{\end{flushleft}}

\newcommand{\bt}{\begin{tabbing}}
\newcommand{\et}{\end{tabbing}}

\newcommand{\beqn}{\begin{align}}
\newcommand{\eeqn}{\end{align}}

\newcommand{\beqs}{\begin{align*}} 
\newcommand{\eeqs}{\end{align*}}  

\newtheoremstyle{Fancyplain}
{\topsep}   
{\topsep}   
{\itshape}  
{0pt}       
{\bfseries} 
{}         
{5pt plus 1pt minus 1pt} 
{\thmname{#1} \thmnumber{#2}. \thmnote{\normalfont\bfseries#3.}}

\theoremstyle{Fancyplain}
\newtheorem{thm}{Theorem}
\newtheorem{lem}{Lemma}

\newtheorem{prop}[lem]{Proposition}

\crefname{thm}{Thm.}{Thms.}
\Crefname{thm}{Theorem}{Theorems}
\crefname{lem}{Lem.}{Lems.}
\Crefname{lem}{Lemma}{Lemmas}
\crefname{cor}{Cor.}{Cors.}
\Crefname{cor}{Corollary}{Corollaries}
\crefname{prop}{Prop.}{Props.}
\Crefname{prop}{Proposition}{Propositions}

\newtheoremstyle{Fancydefinition}
{\topsep}   
{\topsep}   
{\normalfont}  
{0pt}       
{\bfseries} 
{}         
{5pt plus 1pt minus 1pt} 
{\thmname{#1} \thmnumber{#2}. \thmnote{\normalfont\bfseries#3.}}

\theoremstyle{Fancydefinition}
\newtheorem{defn}[lem]{Definition}
\newtheorem{example}{Example}

\newtheorem{rem}{Remark}
\newtheorem{asmp}{Assumption}

\crefname{defn}{Defn.}{Defns.}
\Crefname{defn}{Definition}{Definitions}
\crefname{example}{Ex.}{Exs.}
\Crefname{example}{Example}{Examples}
\crefname{xca}{Ex.}{Exs.}
\Crefname{xca}{Exercise}{Exercises}
\crefname{rem}{Rem.}{Rems.}
\Crefname{rem}{Remark}{Remarks}
\crefname{asmp}{Asmp.}{Asmps.}
\Crefname{asmp}{Assumption}{Assumptions}
\crefname{section}{Sec.}{Secs.}
\Crefname{section}{Section}{Sections}

\numberwithin{equation}{section}
\numberwithin{figure}{section}

\usepackage{comment}

\bibliography{MyLibrary2}

\newcommand{\eps}{\epsilon}
\newcommand{\hx}{\hat{x}}

\allowdisplaybreaks

\usepackage[inactive]{fancytooltips}

\begin{document}


\title{On Unbounded Delays in Asynchronous Parallel Fixed-Point Algorithms\thanks{This work was supported in part by NSF grant ECCS-1462398 and ONR grant N000141712162.}}

\author{Robert Hannah\footnote{Email: \href{mailto:roberthannah89@math.ucla.edu}{RobertHannah89@math.ucla.edu}} }
\author{Wotao Yin\footnote{Email: \href{mailto:wotaoyin@math.ucla.edu}{WotaoYin@math.ucla.edu}}}
\affil{Department of Mathematics, University of California, Los Angeles, CA 90095, USA}

\date{\today}
\maketitle

\begin{abstract}
The need for scalable solvers for massive optimization problems has motivated the development of asynchronous-parallel algorithms, where a set of nodes runs in parallel with little or no synchronization, thus computing with delayed information. This paper develops powerful Lyapunov-functions techniques, and uses them to study the convergence of the asynchronous-parallel algorithm ARock under \emph{potentially unbounded delays}.

ARock is a very general asynchronous algorithm, that takes many popular algorithms as special cases: For instance, asynchronous block gradient descent, forward backward, ADMM, etc. Therefore our results have broad implications, and a range of applications. ARock parallelizes a fixed-point iterations by letting a set of nodes randomly choose solution coordinates to update in an asynchronous parallel fashion. The existing analysis of ARock assumes the delays to be bounded, and uses this bound to set a step size that is important to both convergence and efficiency. Other works, though allowing unbounded delays, impose strict conditions on the underlying fixed-point operator, resulting in limited applications.

In this paper, convergence is established under unbounded delays, which can be either stochastic or deterministic. The proposed step sizes are more practical and larger than those in the existing work. The step size adapts to the delay distribution or the current delay being experienced in the system, instead of being limited by worst-case scenario delays. New Lyapunov functions, which are the key to analyzing asynchronous algorithms, are generated to obtain our results. A general strategy for generating Lyapunov functions is presented, which may find application in convergence analyses of other algorithms. A set of applicable optimization algorithms with large-scale applications are given, including machine learning and scientific computing algorithms.
\end{abstract}

\section{Introduction}
Today there is a great need for efficient algorithms to solve large-scale optimization problems or PDE problems in machine learning, big data, network analysis, cosmology, weather simulations, and other areas. The power of an individual core, after 30 years' exponential growth, has stopped increasing significantly since 2005. Before this point, running the same serial algorithm on the same problem would become faster year-over-year because of the increasing power of individual cores, however this is no longer the case. Moving forward, CPUs will only become faster through the addition of more cores rather than more powerful cores (see \cite{Sutter2005_free,Sutter2011_welcome}). Therefore the only way to utilize more powerful processors is to exploit parallelism. The trade off is that parallel algorithms are difficult to analyze and implement. However, the rewards for success are great: breakthroughs in parallel computing have implications for many other scientific fields.

\subsection{Motivation and importance of asynchronous algorithms}
The vast majority of parallel algorithms are \emph{synchronous} algorithms.
For instance the synchronous-parallel Gauss-Jacobi algorithm divides the problem space $\RR^N$ into $p$ coordinate blocks. At every iteration, these blocks are updated by a corresponding set of $p$ processors, and each processor's update is communicated to every other processor. Synchronous algorithms are simpler to analyze and implement; however, they have major drawbacks, such as synchronization penalty. At each iteration, all processors must wait for the results of the slowest processor to be received.

\subsubsection{Disadvantages of synchronous algorithms}
Synchronous algorithms may become impractical at scale, or on a busy network. Network latency is a major problem and bottleneck for parallel algorithms. Over a 20-25 year period on a wide range of systems, latency has improved by a factor of $20-40$ whereas CPU speeds have improved by a factor of $1000$ \cite{RumbleOngaroStutsmanRosenblumOusterhout2011_its}. This means that synchronizing at every step can be extremely expensive, and the divergence between processing speeds and latency will make this problem worse over time.

Moreover, these modest improvements in latency refer to the hardware's maximum performance. Latency and bandwidth are much worse in large data centers, which are typically very congested: Spikes in traffic can cause latency to increase temporarily by a factor of $20$ \cite{RumbleOngaroStutsmanRosenblumOusterhout2011_its}. Congestion also causes packet loss: Some data may fail to reach all parties, and must be sent again. If any computing node in a synchronous-parallel system experiences congestion or packet-loss, the entire system must wait for that one node. In addition,  dedicated access to computing nodes often cannot be guaranteed. Nodes may suddenly start being used by another user, temporarily go offline, etc. causing further unpredictable delays. The more processors that are used, the higher the likelihood that one will experience problems, and that all others will have to wait.

In addition, sometimes the structure of the problem makes synchronous-parallel solvers inefficient. For instance, it may not be feasible to break a problem into subproblems of equal difficulty. If the computing nodes have roughly equal computational power, nodes that are assigned easier subproblems will frequently be waiting on nodes assigned harder subproblems. The more heterogeneous the difficulty of subproblems, the more problematic this issue becomes.

What is needed is a more flexible framework for parallel optimization: One that is resilient to latency, unpredictable and congested networks, packet loss, heterogeneous subproblem difficulty, and other practical issues.

\subsubsection{Advantages of asynchronous algorithms}
A node in an asynchronous algorithm, instead of waiting to receive results from all other nodes, simply computes its next update using the most recent information it has received. Using outdated information will still result in convergence if the asynchronous algorithm is properly designed.

Latency, congestion,
and random delays will no longer cripple the system, because processors
can make progress without waiting on the results of the slowest processor.
Asynchronous algorithms are resilient to packet-loss, unexpected drains
on computing power, the loss of a node, and many other common problems
on large congested networks. The speed of asynchronous algorithms is more related to
the aggregate computing power and bandwidth of the system, rather than the speed of the slowest processor.

In addition, the algorithm discussed in this paper dynamically balances load with random coordinate block assignment: Processors take on as much work as they are currently able to, and no workload tuning is required.

There is, however, a trade-off: Using outdated information means
the error decreases less per iteration. However more iterations can occur per second because of vastly reduced synchronization penalty. Promising empirical
obtained in \cite{PengXuYanYin2015_arock} suggest that
this trade-off is a favorable one.

\subsection{Fixed-point algorithms}
In this paper, we consider convergence of ARock \cite{PengXuYanYin2015_arock} under possibly unbounded asynchronous delays. We analyze ARock because many popular algorithms are special cases of it: For instance asynchronous block gradient descent, proximal point, forward backward, etc. Therefore our results are very general.

ARock is a fixed-point algorithm. The fixed-point framework is used because it is what allows this generality, and many problems and algorithms can be written in the fixed-point form. Take a \textbf{nonexpansive} operator $T:\HH\to\HH$ (i.e. an operator with Lipschitz constant $L\leq 1$). The aim is to find a \textbf{fixed-point} of this operator: That is, a point $x^*\in\HH$ such that $Tx^*=x^*$. For example, smooth minimization of a convex function $f:\HH\to\RR$ with $L$-Lipschitz gradient $\nabla f$ is equivalent to finding a fixed point of the nonexpansive operator $T=I-\frac{2}{L}\nabla f$, where $I$ is the identity. The set of fixed points of an operator $T$ is denoted $\text{Fix}(T)$

We introduce a classic fixed-point algorithm: the Krasnosel'ski\u{\i}-Mann (KM) algorithm. ARock is essentially an asynchronous block-coordinate version of KM iteration.

\begin{defn}[Krasnosel'ski\u{\i}-Mann algorithm]
Let $\eps>0$, and $\eta^{k}$ be a series of step lengths in $\p{\eps,1-\eps}$. Let $T$ be a nonexpansive operator with at least one fixed point, and
\begin{align}\label{eq:S_def}
S &= I-T.
\end{align}
The KM Algorithm is defined by the following KM iteration: start from $x^0$ and, for $k=0,1,2,\ldots$, do
\begin{align}
\begin{aligned}
x^{k+1} &= x^{k}-\eta^{k}S(x^{k})
\end{aligned}
\end{align}
\end{defn}

KM iteration takes many popular algorithms as special cases. Below in \Cref{tab:Special-cases-KM} we demonstrate how common optimization algorithms, are simply special cases of KM iteration using the appropriate fixed-point operator. Since ARock is simply asynchronous KM, our results will apply to all algorithms in Table \ref{tab:Special-cases-KM}.

In Table \ref{tab:Special-cases-KM}, gradients such as $\nabla f,\nabla g,\nabla h$ are assumed to be Lipschitz continuous with constants $L_{f},L_{g},L_{h}$, respectively. The resolvent operator $J_{\gamma \partial f}=(I+\gamma \partial f)^{-1}$ is equivalent  to the proximal operator of $f$, that is,
\begin{align*}
J_{\gamma \partial f}=(I+\gamma \partial f)^{-1}(y) = \argmin_x f(x) +\frac{1}{2\gamma}\|x-y\|^2,\quad \forall y.
\end{align*}
The reflected resolvent $R_{\gamma\partial f}$ is defined as $2 J_{\gamma \partial f} - I$. Also, it is assumed that $\partial (f+g) = \partial f+ \partial g$ and $\partial (F+G) = \partial F+ \partial G$.

\medskip
\begin{table}
\centering
\caption{Selection of common algorithms that are special cases of KM iteration, and their corresponding fixed-point operator.\label{tab:Special-cases-KM}}
\begin{tabular}{>{\centering}p{3cm}|c|>{\centering}p{4cm}|>{\centering}p{2.5cm}}
        \hline
        \textbf{Optimization problem} & \textbf{Algorithm} & \textbf{Nonexpansive fixed-point operator $T$} & \textbf{Assumption}\tabularnewline
        \hline
        \hline
        $\min\,f\p{x}$ & Gradient descent & $I-\gamma\nabla f$ & 
$\gamma\in(0,\frac{2}{L_{f}}]$\tabularnewline
        \hline
        $\min\,f\p{x}$ & Proximal point & $J_{\gamma \partial f}$ & $\gamma>0$\tabularnewline
        \hline
        $\min f\p{x}+g\p{x}$ & Forward backward & $J_{\gamma \partial f}\circ\p{I-\gamma\nabla g}$ & 
$\gamma\in(0,\frac{2}{L_{g}}]$\tabularnewline
        \hline
$\min\{g\p{x}:x\in C\}$ & Projected gradient & $\mathrm{Proj}_C\circ\p{I-\gamma\nabla g}$ & $\gamma\in(0,\frac{2}{L_{g}}]$ \tabularnewline
        \hline
        $\min f\p{x}+g\p{x}$ & Peaceman-Rachford & $R_{\gamma\partial f}\circ R_{\gamma\partial g}$ & 
$\gamma >0$\tabularnewline
        \hline
        $\min \sum_{i=1}^d f_i\p{x}$ & Parallel Peaceman-Rachford & $(\frac{2}{d}\mathbf{1}\mathbf{1}^T-I)\circ R_{\gamma\partial \mathbf{f}}$ where $\mathbf{f}=[f_1;\ldots;f_d]:\HH^d\to\RR^d$ & $\gamma >0$\tabularnewline
        \hline
        $\min f\p{x}+g\p{x}$ & Douglas-Rachford & $\frac{1}{2}I+\frac{1}{2}R_{\gamma\partial f}\circ R_{\gamma\partial g}$ &
$\gamma >0$\tabularnewline
        \hline
        $\min f\p{x}+g\p{x}+h\p{x}$ & Davis-Yin & $I-J_{\gamma \partial g} +J_{\gamma \partial f}\circ (2J_{\gamma \partial g}-I-\gamma \nabla h\circ J_{\gamma \partial g})$ & $\gamma \in(0,\frac{2}{L_{h}}]$\tabularnewline
        \hline
        $\min\{f(x)+g(z):Ax+Bz=b\}$ & ADMM & $\frac{1}{2}I+\frac{1}{2}R_{\gamma\partial F}\circ R_{\gamma\partial G}$, where $F(y):=f^*(A^Ty)$, $G(y):=g^*(B^Ty)-b^Ty$ & 
$\gamma >0$
\tabularnewline
        \hline
\end{tabular}
\end{table}
\smallskip

\begin{rem}[Interpretation]
Columns 1 and 2 contain the optimization problem and the a common algorithm used to solve it. Column 3 gives the nonexpansive fixed-point operator $T$ corresponding to the algorithm. A fixed point of $T$ corresponds to a solution to the original optimization problem. When you apply the KM iteration to $T$, you obtain the algorithm in column 2. Column 4 contains assumptions necessary for convergence. The derivations of the algorithms and operators, as well as the proof of nonexpansiveness, are out of the scope of this paper. We refer the interested reader to \cite{BauschkeCombettes2011_convex,DavisYin2015_threeoperator}.

\end{rem}

\begin{rem}[Adding constraints]
Constraints can be introduced into the optimization problem with the addition of indicator functions\footnote{The indicator function $\iota_C(x)$ is equal to $0$ if $x\in C$, and $\infty$ otherwise.}. For example, $\min_{x\in C} f(x)$ is equivalent to $\min_{x\in \RR^N} f(x)+g(x)$ for $g(x)=\iota_{C}(x)$. Of course this restricts you to algorithms that do not assume $g(x)$ is differentiable. But algorithms that use the resolvent are possible, since $J_{\gamma \partial f}(y)$ equals the projection of $x$ onto $C$. For example, the projected gradient algorithm is simply forward-backward with $g(x)=\iota_C(x)$.

\end{rem}

\begin{example}[Gradient descent]
As an example, we prove the first row: That is, gradient descent is simply KM on the operator $I-(2/L_f)\nabla f$. By the Baillon-Haddad theorem, $\frac{1}{L_{f}}\nabla f$ is firmly nonexpansive, and therefore $T=I-\frac{2}{L_{f}}\nabla f$
is a nonexpansive operator. In addition we have:
\begin{align*}
        x^{*} \in\text{Fix}\p{T}
        \iff & \p{I-\p{2/L_f}}(x^{*}) =x^{*}\\
        \iff & \p{2/L_f}\nabla f(x^{*})  =0\\
        \iff & x^{*}~\text{is a minimizer of}~f.
\end{align*}
Therefore the corresponding fixed point problem for $T=I-(2/L_f)\nabla f$ is equivalent to the function minimization problem. Applying KM iteration to this $T$ with step size $\eta^{k}$ yields:
\begin{align*}
        x^{k+1} & = x^{k}-\frac{2\eta^{k}}{L_{f}}\nabla f(x^{k}),
\end{align*}
which is the gradient descent algorithm.

\end{example}

\subsection{The ARock algorithm}
We finally define the ARock algorithm. Take a space  $\HH$  on which to solve an optimization problem.   $\HH$  can be the real space $\RR^N$ or, more generally, a separable Hilbert space.
 Break this space into $m$ orthogonal subspaces: $\HH = \HH_{1} \times\ldots\times\HH_{m}$ so that vectors $x\in\HH$ can be written as $(x_1,x_2,\ldots,x_m)$  where each $x_i$  is $x$'s component in subspace $\HH_i$. Take a nonexpansive operator $T:\HH\to\HH$. Let $S=I-T$ and $Sx=\p{S_{1}x,\ldots,S_{m}x}$ where $S_{j}x$ denotes the $j$'th block of $Sx$.

\textbf{Convention:} \emph{Superscripts} will denote the iteration
number of a sequence of points $x^0,x^1,x^2,\ldots$. \emph{Subscripts} will denote different blocks of a vector or operator, e.g., $x=(x_1,x_2,\ldots,x_m)$ and $Sx=\p{S_{1}x,\ldots,S_{m}x}$. {For instance, $x^k_l$ is the $l$th block of iterate $x^k$. $S_l x^k$ is the $l$th block of $S(x^k)$.}

\begin{defn}[The ARock Algorithm]\label{def:ARock}
Let $\eta^0,\eta^{1},\eta^{2},\ldots$ be a series of step lengths. Let $T$ be a nonexpansive operator with at least one fixed point $x^*$, and $S=I-T$. Take a starting point $x^0\in\HH$. Then the ARock algorithm \cite{PengXuYanYin2015_arock} is defined
via the iteration: start from $x^0$ and, for $k=0,1,2,\ldots$, do

\begin{equation}\label{eq:def:ARock}
\mbox{for}~i=1,\ldots,m,\quad x_i^{k+1}  \gets
\begin{cases}
x_i^{k}-\eta^{k}S_{i}(\hx^{k}),& i=i\p{k},\\
x_i^{k},& i\neq i\p{k},
\end{cases}
\end{equation}
where the \textbf{delayed iterate} $\hx^{k}$ represents a possibly outdated version of the iteration vector $x^k$ and the
\textbf{block index sequence} $i(k)$ specifies which block of $x^k$ is being updated to produce the next iterate $x^{k+1}$.
\end{defn}

The ARock algorithm resembles KM iteration. However we use a delayed iterate $\hx^k$ because of asynchronicity, and we update one block at a time. $\hat{x}^k$ and $i(k)$ will be defined precisely in \Cref{sec:Setup}.

\subsection{Setup} \label{sec:Setup}

We now precisely define the block sequence $i(k)$, and the delayed iterate $\hx^k$, that we will use in this paper.

\subsubsection{Block sequence}\label{sec:Block-sequence}

\begin{asmp}[IID block sequence]\label{asmp:Block-sequence}
The sequence in which blocks of the solution vector are updated, $i(k)$, is a series of uniform IID random variables that take values $1,2,\ldots,m$ each with probability $1/m$.
\end{asmp}

A uniform distribution is \emph{not} strictly necessary, but is simpler. This assumption will hold if we allow all nodes to randomly update any block chosen in a uniform IID fashion (assuming all blocks have update times with the same distribution). However this may result in bad data locality\footnote{That is, implementing the algorithm in this way may require a lot of data movement. This is because every single time a node makes an update, it must be sent the data for the entire block that it is updating.}. An alternative is to assign the $m$ computing nodes one of $m$ blocks each, and assume that the times taken to compute updates follow IID Poisson processes \cite{PengXuYanYin2016_convergence}. Future work may involve weakening this assumption, perhaps extending the result beyond Poisson distributions.

\subsubsection{Delayed iterates}\label{sec:Delayed-iterates}
Let $\vec{j} = (j_1,\ldots,j_m) \in \NN^m$ be a vector, and $x^0,x^1,x^2,\ldots$ a series of iterates. Let $k\in\NN$ be the iteration number. We find it convenient to define:
\begin{align}
x^{k-\vec{j}} & = \p{x_1^{k-j_{1}},x_2^{k-j_{2}},\ldots,x_m^{k-j_{m}}}.
\end{align}
We define a series of \textbf{delay vectors} $\vec{j}(0), \vec{j}(1),\vec{j}(2),\ldots$ in $\NN^m$, corresponding to $x^0, x^1, x^2,\ldots$ respectively. The components of these delay vectors are as follows:
\begin{align}
\vec{j}\p{k} &= \p{j\p{k,1},j\p{k,2},\ldots,j\p{k,m}}.
\end{align}
The \textbf{current delay} is defined as\footnote{\textbf{Note:} The lack of the vector symbol distinguishes the current delay from the delay vector.}:
\begin{align}\label{eq:def:current-delay}
j\p{k} &= \max_{1\le i\le m}\cp{ j\p{k,i}}.
\end{align}
Using this, we define the delayed iterate.
%
\begin{defn}[Delayed iterate]
The delayed iterate $\hx^k$ is defined as\footnote{\textbf{Stronger asynchronicity:} It is possible to have more general asynchronicity, where different components of the \textit{same block}, $x_l\in\HH_l$, have different ages. This leads to similar results, and a similar proof, but the current setup was chosen for simplicity.}:
\begin{align}
\hx^k &= x^{k-\vec{j}\p{k}},~\mbox{or equivalently,}\\
\hx^k= \p{\hx_{1}^{k},\hx_{2}^{k},\ldots,\hx_{m}^{k}} &= \p{x_{1}^{k-j\p{k,1}},x_{2}^{k-j\p{k,2}},\ldots,x_{m}^{k-j\p{k,m}}}.
\end{align}
\end{defn}

Recall that asynchronous algorithms do not wait to receive results from all other nodes, but simply perform their updates with the most recent information they have available. Therefore processors may not necessarily have the most up-to-date information on $x^k$, but instead have a delayed iterate $\hx^k$. Every block of $\hx^{k}$ is potentially outdated by a different amount (that is, inconsistent reads are possible). $j\p{k,i}$ denotes how many iterates out of date block $i$ is at step $k$. Block $1$ may be up to date, so $j\p{k,1}=0$. Block $5$ may be $17$ iterations behind, so $j\p{k,5}=17$.

Clearly this is a very general model: There is a series of delay vectors $\vec{j}(0), \vec{j}(1), \vec{j}(2),\ldots$ that represents how old the information that a computing node has access to is. How these delay vectors are determined depends on the model of asynchronicity chosen. We consider two possibilities in this paper: stochastic and deterministic delays (see \Cref{sec:Stochastic-unbounded-delay,sec:Deterministic-unbounded-delay} respectively).

\subsection{New results and contributions} \label{sec:New-results-and-contributions}
The contributions of this paper are two-fold. First, we prove the convergence of ARock under unbounded delays that are either stochastic or deterministic. This is achieved by constructing and analyzing Lyapunov functions. The second contribution of this paper is to describe and demonstrate general techniques for constructing Lyapunov functions, which appears to be the key to analyzing the convergence of asynchronous algorithms, and many other types of algorithms.

We leave coding and numerical tests to our future work because they involve engineering issues that are beyond the scope of this work. For example, the current delay, which affects the step size, can be obtained by many different methods. Our ongoing work such as \cite{EdmundsPengYin2016_tmac} will develop codes and numerical results. There is however a recent implementation of asynchronous-parallel stochastic coordinate descent in C \cite{Huang2016_asynml}.

In rest of this subsection, we present these convergence results, but not in their most general forms. 
A more complete description of these results in all their generality is given in \Cref{sec:Proof-of-convergence-stochastic,sec:Proof-of-convergence-deterministic}.

\subsubsection{Stochastic unbounded delay} \label{sec:Stochastic-unbounded-delay}
The first result is the convergence of ARock under stochastic, potentially unbounded
delays. First we precisely define the assumptions on the delay:

\begin{defn}[Evenly old delays]\label{def:Evenly-old}
We say that delays are ``evenly old'' if there exists some constant $B$ such that, with probability $1$, we have $\abs{j\p{k,i}-j\p{k,l}} \leq B$ for all $k\in \NN$, $1\leq i \leq m$, and $1\leq l \leq m$.

\end{defn}

Delays can be arbitrarily large, but the ages of the various block are similar if they are evenly old. Clearly, if we have \textbf{bounded delay} (that is, with probability $1$ we have $j(k)\leq\tau$ for some $\tau$), this implies the evenly old property with constant $B=\tau$.

\begin{asmp}[Stochastic unbounded delays] \label{asmp:Stochastic-delays}
The sequence of delay vectors $\vec{j}\p{0},\vec{j}\p{1},\ldots$ are IID, and independent of the block sequence $i\p{0},i\p{1},\ldots$. In addition, they are evenly old.
\end{asmp}
Hence there exists a function $p:\NN^{m}\to\sp{0,1}$ such that, for all $k\in\NN$, the probability that $\vec{j}(k)$ equals some vector $\vec{v}$ is given by
\begin{align}
\PP\sp{\vec{j}\p{k}=\vec{v}} &= p\p{\vec{v}}.
\end{align}
Define
\begin{align}
P_{l} &= \PP\sp{j\p{k}\geq l}.
\end{align}
%

\begin{thm}[Convergence under stochastic unbounded delays] \label{thm:Convergence-stochastic-delays-intro}
Assume that the block sequence $i(k)$ is a uniform IID block sequence (\Cref{asmp:Block-sequence}) and that the delays vectors $\vec{j}(k)$ are an evenly old, IID sequence that is independent of the block sequence (\Cref{asmp:Stochastic-delays}). Let the step size be $\eta^k = c h$ for an arbitrary fixed\footnote{By ``arbitrary fixed'' we mean that the constant $c$ can be any number in $(0,1)$, so long as that number does not change. However it is possible to relax this.} $c\in\p{0,1}$, and $h$ given below. Then the iterates of ARock  converge weakly to a solution with probability $1$ if either of the following holds:

1. $\sum_{l=1}^{\infty}\p{lP_{l}}^{1/2}<\infty$,
 and setting $h=\p{1+\frac{1}{\sqrt{m}}\sum_{l=1}^{\infty}P_{l}^{1/2}\p{l^{1/2}+l^{-1/2}}}^{-1}$.

2. $\sum_{l=1}^{\infty}P_{l}^{1/2}l<\infty$, and setting $h=\p{1+\frac{2}{\sqrt{m}}\sum_{l=1}^{\infty}P_{l}^{1/2}}^{-1}$.

\end{thm}

Convergence under unbounded delays in this setting has only been proven under very strong assumptions (See \Cref{sec:Related-work} for a discussion of existing results). Additionally, this result improves on the step size criterion of ARock and other similar algorithms if we are willing to assume stochastic delays (e.g. \cite{PengXuYanYin2015_arock,LiuWright2015_asynchronous,LiuWrightReBittorfSridhar2015_asynchronous}). So for instance, there may be a scenario where the maximum delay $\tau$ is very high, but delays near that size rarely occur. \Cref{thm:Convergence-stochastic-delays-intro} implies that asynchronous algorithms will convergence under a much larger step size than prior work. From this it can be argued that the Lyapunov function techniques introduced in this paper are very efficient, and may perhaps be used to obtain tighter convergence rates when applied to problems such as function minimization, rather than a general fixed-point problem.

Table \ref{tb:par} gives some example distributions, and corresponding values $P_l$ and step size $h$ (we only used the second $h$ since it is easier to calculate). We give an upper bound for $P_l$ and lower bound for $h$ to simplify expressions. Let $I[A](x)$ denote the characteristic function (i.e. a function that equal $1$ for $x\in D$ and $0$ otherwise).

\begin{table}
\centering
\caption{Example delay distributions and step sizes\label{tb:par}}
\begin{tabular}{ c | c | c }
\hline
\textbf{Distribution of $j(k)$}

&

\textbf{$P_l$ Upper bound}

&

\textbf{$h$ Step size lower bound}

\\ \hline\hline

$j(k)$ arbitrary, with $\leq \tau$

&

$1\cdot I[0\leq l \leq \tau](l)$

&

$\p{1+\frac{2\tau}{\sqrt{m}}}^{-1}$

\\ \hline

$j(k)$ uniform on $\cp{0,1,2,\ldots,\tau}$

&

$\p{1-\frac{l}{\tau+1}}I[0\leq l \leq \tau]$

&

$\p{1+\frac{4\tau}{3\sqrt{m}}}^{-1}$

\\ \hline

$\begin{array}{c}
\text{$j(k)$ exponentially decaying.}\\
\text{I.e. $\PP[j(k)=l]\leq C r^{l}$ for $1>r>0$.}
\end{array}$

&

$C\frac{r^{l/2}}{1-r}$

&

$\p{1+2\sqrt{\frac{C}{m}}\frac{r^{1/2}}{\p{1-r^{1/2}}^{3/2}}}^{-1} $

\\ \hline

Each of $p$ agents has update time $Y\in[a,b]$ 

&

*

&

$\p{1+\frac{2p\cdot \p{\frac{b}{a}+1}}{\sqrt{m}}}^{-1}$

\\ \hline

\end{tabular}
\end{table}

In addition to example distributions, we consider the step size in the following scenario. Let $Y$ be a random variable representing the time between when a node starts reading the solution vector $x$, and when its update is applied. Say that we have $a\leq Y \leq b$, and that there are $p$ computing nodes. In the worst-case scenario, a node takes $b$ seconds, and $p\cdot \p{\frac{b}{a}+1}$ updates have occurred during this time. Hence ignoring the specifics of the distribution, we have a delay bound $\tau=p\cdot \p{\frac{b}{a}+1}$. It can be seen from Table \ref{tab:Special-cases-KM}, that in this scenario, if $b/a$ doesn't grow, then $\sqrt{m}\gg p$ implies a step size of $c\in (0,1)$ will result in convergence.

\begin{rem}[Step size heuristic]
Even if the assumption of independent IID delays does not hold in practice, the preceding step size gives a useful heuristic to use given an empirical distribution of delays measured in a system. For example, when the number of blocks $m$ satisfies $\sqrt{m}\gg\sum_{l=1}^{\infty}P_{l}^{1/2}$, the step size sequence should be $\eta^k\approx c$, where $c\in(0,1)$ is an arbitrary fixed constant.

\end{rem}

\subsubsection{Deterministic unbounded delay} \label{sec:Deterministic-unbounded-delay}
The second result of this paper proves convergence of ARock and related algorithms under deterministic unbounded delays. In order to achieve convergence, it is necessary to use a step size $\eta^k$ that is a decreasing function of the current delay $j\p{k}$ (whereas in \Cref{thm:Convergence-stochastic-delays-intro}, a constant step size was sufficient). Also convergence is only on a family of subsequences.

\begin{asmp}[Deterministic unbounded delays]\label{asmp:Deterministic-delays}
The sequence of delay vectors $\vec{j}(0), \vec{j}(1), \vec{j}(2),\ldots$ is an arbitrary sequence in $\NN^m$, independent of $i(k)$, with $\
 j\p{k}<\infty$.
\end{asmp}

\begin{defn}[Convergence on subsequences of bounded delay] \label{def:Convergence-subsequences-bd-delay}
Let $x^0, x^1, x^2,\ldots$ be a sequence of iterates and $\vec{j}(0), \vec{j}(1), \vec{j}(2),\ldots$ a corresponding sequence of delay vectors, with $\liminf j\p{k}<\infty$. Let $Q_J$ be the subsequence of $x^0, x^1, x^2,\ldots$ where the iterates $x^k$ with current delay $j(k)>J$ are removed\footnote{$Q_J$ represent subsequences of bounded delay.}. We say that $x^k$ converges to $x^*$ on subsequences of bounded delay if $x^k$ converges to $x^*$ on every subsequence $Q_J$ for $J\geq\liminf j\p{k}$\footnote{$J\geq\liminf j\p{k}$ ensures that
$Q_J$ is an infinite subsequence.}.

\end{defn}

\begin{thm}[Convergence under deterministic unbounded delays] \label{thm:Convergence-deterministic-delays-intro}
Assume that the block sequence $i(k)$ is a sequence of uniform IID random variables (\Cref{asmp:Block-sequence}) and that the sequence of delay vectors $\vec{j}(0), \vec{j}(1), \vec{j}(2),\ldots$ is an arbitrary sequence in $\NN^m$, independent of $i(k)$, with $\liminf j\p{k}<\infty$ (\Cref{asmp:Deterministic-delays}). Pick arbitrary, fixed $c\in\p{0,1}$ and $\gamma>0$. Let the step size be
\begin{align}
\eta^{k} & = c \p{1+\frac{1}{\sqrt{m}}\p{1+\frac{1}{\gamma}+\frac{1}{2+\gamma}\p{j(k)+1}^{2+\gamma}}}^{-1}.\label{eq:Gamma-step-size-deterministic}
\end{align}
Then with probability $1$, the iterates of ARock weakly converge to a solution $x^*$ on all subsequences of bounded delay $Q_J$ for $J\geq\liminf j\p{k}$ (\Cref{def:Convergence-subsequences-bd-delay}), where $x^*$ does not depend on the bound $J$.
\end{thm}

This step size rule assumes a worst case scenario. In practice it can be used if it was necessary to be certain that the algorithm converges. Even if network conditions are very unfavorable, making delays large, the algorithm with the step size \eqref{eq:Gamma-step-size-deterministic} makes some progress at every step. This result could also be used in the bounded delay regime when the bound $\tau$ is not known in advance. In previous results, $\tau$ is needed in advance to calculate the correct step size.

\Cref{thm:Convergence-deterministic-delays-intro} also provides a rule adaptive to the current delay. If the step size were set according to $\tau$ (which is the case for the vast majority of recent papers), the step size may be exceedingly pessimistic if a delay of $\tau$ is very rare. However our result implies a much larger allowable step size when delays are smaller (even if they may becomes large at some point in the future). When the delays are bounded (but the bound is possibly unknown to us), \Cref{thm:Convergence-deterministic-delays-intro} implies weak convergence of the full sequence with probability $1$, not merely on subsequences of bounded delay.

The step size rule also gives the following useful heuristic: When the number of blocks $m$ satisfies $\sqrt{m}\gg\p{j\p{k}+1}^{2+\gamma}$, the step size should be $\eta^k\approx c\in (0,1)$.

\subsection{Applications}

\begin{table}
\caption{Applications of ARock\label{tab:Applications-of-ARock}}
\small
\begin{tabular}{|>{\centering}p{2.7cm}|>{\centering}p{3.8cm}|>{\centering}p{8.3cm}|}
        \hline
        \textbf{Convex Optimization Problem} & \textbf{Setup} & \textbf{ARock Iteration}\tabularnewline
        \hline
        \hline
        \textbf{Smooth minimization:}  min $f\p{x}$ & $\nabla f$ is $L$-Lipschitz,
        $\nabla f=\begin{pmatrix}
                \nabla_1 f\\
                \vdots\\
                \nabla_m f
            \end{pmatrix}
  $
& $x_{i_{k}}^{k+1}\leftarrow x_{i_{k}}^{k}-\frac{2\eta^k}{L}\nabla f_{i_k}(\hx^{k})$\tabularnewline
        \hline
        \textbf{Constrained minimization:} min $f\p{x}$ subject to $\ell\le x\le u$  & same as above
& $x_{i_{k}}^{k+1}\leftarrow
x_{i_{k}}^{k}-\eta^k\bigg(\hat{x}^k_{i_k}-\mathrm{Proj}_{[\ell_{i_k},u_{i_k}]}\big(\hat{x}^k_{i_k}-\frac{2}{L}\nabla f_{i_k}(\hx^{k})\big)\bigg)$\tabularnewline
        \hline
  \textbf{Composite minimization (ERM model):}
  min $f\p{x}+g\p{x}$ &
  same as above, plus $g\p{x}=\sum_{i=1}^mg_i\p{x_i}$
  &
  $x_{i_{k}}^{k+1}\leftarrow$ $x_{i_{k}}^{k}-\eta^k\bigg(\hat{x}^k_{i_k}-\mathrm{prox}_{\frac{2}{L}g_i}\big(\hat{x}^k_{i_k}-\frac{2}{L}\nabla f_{i_k}(\hx^{k})\big)\bigg)$\tabularnewline
                \hline
  \textbf{Kernel SVM:}
  $\min_s\,\frac{1}{2}s^TQ s -e^Ts$ subject to $\sum_{i} y_is_i=0$, $0\le s_i\le C,~\forall i$ &
  training set $\{x_i,y_i\}$, $y_i\in\{\pm 1\}$, kernel $k(\cdot,\cdot)$, $Q_{ij}=y_iy_jk(x_i,x_j)$, applies Davis-Yin
  &
  See the last equation in \cite[Section 5.2.1]{PengWuXuYanYin2016_coordinate}, and apply it with damping $\eta^k$
  \tabularnewline
                \hline
        \textbf{Linear System:}

                Solve $Ax=b$  & $A$ is symmetric positive definite,
                $\begin{pmatrix}
                        \text{--}~A_{1}~\text{--}\\
                        \vdots\\
                        \text{--}~A_{m}~\text{--}
                        \end{pmatrix}
      x
      =\begin{pmatrix}
                        b_{1}\\
                        \vdots\\
                        b_{m}
                        \end{pmatrix}$
& $x_{i_{k}}^{k+1}\leftarrow x_{i_{k}}^k-\p{\frac{2\eta^{k}}{M}}\p{A_{i_k}\hx^{k}+b_{i_k}}$\tabularnewline
                \hline
                \textbf{Linear System:}
                Solve $Ax=b$ & $A=D+R$ where

                $D$ is diagonal, $M$ off-diagonal,

                $\rho\p{-D^{-1}R}\leq1$ & $x_{i_{k}}^{k+1}\leftarrow x_{i_{k}}^{k}-\eta^{k}\p{\p{I+D^{-1}M}\hx^k-D^{-1}b}_{i_{k}}$\tabularnewline
                \hline
        \end{tabular}
\end{table}

As mentioned, ARock takes a wide variety of algorithms as special cases, such as gradient descent, proximal point, Douglas-Rachford (and Peaceman-Rachford), forward-backward, ADMM, etc. For ARock to be \textit{practical} however, it needs to be possible to efficiently parallelize the corresponding serial iteration. For instance, ARock for smooth minimization is just asynchronous block gradient descent: $x_{i_k}^{k+1}=x_{i_k}^{k}-\frac{\eta^{k}}{L}\nabla_{i_{k}}f\p{\hx^{k}}$. If it is not significantly easier to calculate $\nabla_{i_{k}}f$ than to calculate the full gradient, then ARock is impractical, and parallelization may yield no speedup.

However ARock is practical for a wide variety of algorithms and applications; see the paper \cite{PengWuXuYanYin2016_coordinate} for the structures of operators that give rise to parallelizable ARock algorithms. We present a small sample of applications in \Cref{tab:Applications-of-ARock} (more applications are found in \cite{PengXuYanYin2015_arock,PengWuXuYanYin2016_coordinate}).


In Table \ref{tab:Applications-of-ARock}, $\mathrm{Proj}_{[\ell_{i_k},u_{i_k}]}$ projects a scalar to the interval $[\ell_{i_k},u_{i_k}]$. Each algorithm in column 3, is simply an example of ARock with the appropriate fixed-point operator.

\subsection{Related work}\label{sec:Related-work}
Asynchronous algorithms were first proposed by Chazan and Miranker in \cite{ChazanMiranker1969_chaotic} to solve linear systems. Since then, asynchronous algorithms have been applied to many fields including nonlinear systems, differential equations, consensus problems, and optimization.

Until relatively recently, authors assumed a deterministic sequence of block updates: $i\p{1},i\p{2},\ldots$ with very little restriction. However, this imposes stronger restrictions on the problem. The delays $\vec{j}(k)$ are usually also assumed to be deterministic, but this appears to be relatively less restrictive. In \cite{BertsekasTsitsiklis1997_parallel}, the authors describe two basic classes of deterministic asynchronous scenarios that appeared in the literature.

\begin{defn}[Totally asynchronous iteration]
Every block, $x_{i}$, is updated infinitely many times. Information from iteration $k$ (i.e. the components of $x^{k}$) is only used a finite number of times.

\end{defn}
Total asynchronicity is a very weak condition that leads to convergence results with limited applicability (though there do exist applications to linear problems and strictly convex network flow problems \cite{BertsekasTsitsiklis1997_parallel,TsengBertsekasTsitsiklis1990_partially}). For instance, asynchronous linear iteration $x\mapsto Ax+b$ will only converge in general if the largest eigenvalue of $\abs{A}$ (the matrix obtained by taking an absolute value of every entry) is strictly less than $1$ 
(\cite{ChazanMiranker1969_chaotic,BertsekasTsitsiklis1997_parallel}).

\begin{defn}[Partially asynchronous iteration]
There exists an integer $B$ such that every component, $x_{i}$, is updated at least once every $B$ steps; and the information used by the processors cannot be older than $B$ steps (bounded delay).
\end{defn}

Partially asynchronous algorithms have better convergence properties.
For instance, from \cite{Tseng1991_rate}:

\begin{thm}
For strongly convex $f$ with $\nabla f$ Lipschitz, there is a step size $\gamma_{1}$ such that for any step size $0<\gamma<\gamma_{1}$, asynchronous gradient descent with partial asynchronicity converges at least linearly to a minimum, with rate $\cO\p{\p{1-c\gamma}^{k}}$
for some constant $c$.

\end{thm}

However, the formulas for $c$ or $\gamma_{1}$ are complicated, and the authors did not include them. These constants are also tiny, because one needs to assume the worst-case scenario. The maximum delay $B$ needs to be known in advance to determine the step size.

Stochastic asynchronous algorithms began to appear recently, a popular example being ``Hogwild!'' \cite{RechtReWrightNiu2011_hogwild}. These algorithms always assume a bounded delay ($j\p{k,i}\leq\tau$ for all $k$ and $i$), and that the sequence of blocks $i\p{k}$ is chosen independently and identically with $\PP\sp{i\p{k}=j}=p_{j}$ for fixed nonzero probabilities $p_{j}$. In \cite{AvronDruinskyGupta2014}, the authors prove linear convergence for an asynchronous stochastic linear solver. In \cite{LiuWrightReBittorfSridhar2015_asynchronous}, the authors prove function-value convergence for asynchronous stochastic coordinate descent. Under the assumption that the step size exponentially decays in $\tau$ in a certain way, they prove $\cO\p{1/k}$ convergence for $f$ convex with $\nabla f$ Lipschitz, and linear convergence when $f$ is also strongly convex. This was extended in \cite{LiuWright2015_asynchronous} to composite objective functions. However point convergence ($x^k\to x^*$) is not attained for the non-strongly convex case in these papers\footnote{In the non strongly-convex case, point convergence is stronger than function-value convergence (the former gives the latter, but not vice versa). In the strongly-convex case, they are equivalent.}. The work presented in this paper generalizes and strengthens results from these recent papers on stochastic asynchronous algorithms.

There are recent unbounded delay results in the stochastic unconstrained convex optimization setting \cite{DuchiChaturapruekRe2015_asynchronous,SraYuLiSmola2015_adadelay,AgarwalDuchi2011_distributed}. It is hard to compare results from a different optimization setting. However we note the following: We obtain point convergence ($x^k\rightharpoonup x^*$) rather than function-value convergence ($f\p{x^k}\to f\p{x^*}$) for convex $f$ that is not necessarily strongly convex. The deterministic unbounded delay criterion in \Cref{thm:Convergence-deterministic-delays-intro} is weaker than all other delay assumptions. The step size in these papers converges to $0$ as $k\to\infty$, which is an inevitable part of the problem setting. This makes asynchronicity error less of a problem. Nonetheless, in this paper, we are able to prove convergence in our setting with a step size rule that is only a function of the delay distribution despite unbounded delays (\Cref{thm:Convergence-stochastic-delays-intro}). The step size rule is invariant in $k$, and does not converge to $0$. \Cref{thm:Convergence-deterministic-delays-intro} features a step size that adapts to current delay conditions, once again invariant in $k$, which is cited as a key advantage of \cite{SraYuLiSmola2015_adadelay}.

Our result in \Cref{thm:Convergence-deterministic-delays-intro} can be seen as a halfway point between partial and total asynchrony. Using a slightly stronger assumption than total asynchronicity, we are able to prove a much stronger convergence result.

\subsection{Structure of the paper}
The remainder of the paper is organized as follows. \Cref{sec:Proof-of-convergence-stochastic,sec:Proof-of-convergence-deterministic} give the convergence proofs for the stochastic and deterministic cases, respectively. In \Cref{sec:General-strategy}, we describe a general strategy for generating Lyapunov functions, which appear to be the key to analyzing asynchronous algorithms (as well as many others).

\section{Proof of Convergence for Stochastic Unbounded Delays} \label{sec:Proof-of-convergence-stochastic}
This section proves  \Cref{thm:Convergence-stochastic-delays-full} below, which is a more general version of \Cref{thm:Convergence-stochastic-delays-intro} from the introduction. \Cref{thm:Convergence-stochastic-delays-full} involves a sequence of arbitrary parameters $\eps_{1},\eps_{2},\ldots$ that appear naturally in our analysis. The values of these parameters can be chosen situationally to obtain different result. In \Cref{sec:Parameter-choice-stochastic}, we select (i) the values that give the weakest conditions on delays, and (ii) the values that give the largest allowable step size to obtain the two parts of \Cref{thm:Convergence-stochastic-delays-intro} from the  introduction.

\begin{defn}[Summable sequence]
Let $a=\p{a_1,a_2,\ldots}$ ($a_i\in\RR, \forall i$) be a sequence. $a$ is said to be \textbf{summable} or ``in $\ell^1$'' if its $\ell^1$ norm is finite, that is,
\begin{align*}
\n{a}_{\ell^1} &= \sum^\infty_{i=1} \abs{a_i}<\infty.
\end{align*}
\end{defn}

\begin{thm}[Convergence under stochastic delays] \label{thm:Convergence-stochastic-delays-full}
Consider ARock under the following conditions:

1. The block sequence $i(k)$ is a uniform IID block sequence (\Cref{asmp:Block-sequence}).

2. The  sequence of delay vectors $\vec{j}(k)$ is an evenly old, IID sequence that is independent of the sequence $i(k)$ (\Cref{asmp:Stochastic-delays}).

3. Let $\eps_{1},\,\eps_{2},\ldots\in(0,\infty)$ be an arbitrary sequence of parameters such that $\sum_{i=1}^\infty \frac{1}{\eps_{i}}<\infty$ and $\sum_{l=1}^{\infty}\eps_{l}P_{l}l<\infty$
for $P_{l}=\PP\sp{j\p{k}\geq l}$ (\Cref{asmp:Coefficient-formula-stochastic}).

4. The step size is chosen as $\eta^{k}=ch$ for an arbitrary fixed $c\in\p{0,1}$
and $h=\p{1+\frac{1}{m}\sum_{l=1}^{\infty}\eps_{l}P_{l}+\n{ \frac{1}{\eps_{i}}} _{\ell^{1}}}^{-1}$.

Then  with probability $1$, the sequence of ARock iterates converges weakly to a solution.

\end{thm}

This theorem is proven in \Cref{sec:sub:sub:Proof-of-theorem-stochastic} after we build up a series of results throughout this section. This section is written in a way that attempts to explain the logic and intuition behind the approach taken. A general strategy for constructing Lyapunov functions is presented in \Cref{sec:General-strategy}. In \Cref{sec:Bounded-delay-stochastic}, we discuss how to modify the proof for the simpler case of bounded delay.

\subsection{Proof outline}
Both convergence proofs rely on the following convergence criterion for fixed-point algorithms (see \cite{BauschkeCombettes2011_convex}):

\begin{prop}[Convergence of nonexpansive fixed-point iterations]\label{prop:Convergence-of-fixed-point-algorithms}
Let $T$ be a nonexpansive operator with at least one fixed point. If we have the following:

\normalfont{\textbf{(1) Norm convergence:}}\footnote{We call this property \textit{norm convergence.} The \textit{distance} of $x^k$ to each fixed-point $x^*$ is what is converging (in general to a nonzero value) and not $x^k$ itself.  \textit{\emph{This property does not appear to have been given a name in the literature, although
it is an important property in convergence proofs.}}} $\n{ x^{k}-x^{*}} $
converges for every $x^{*}\in\text{Fix}\p{T}$, and

\normalfont{\textbf{(2) Fixed-point-residual (FPR) strong convergence:}}\footnote{The \textit{fixed-point residual} (FPR) at $x$ is defined as $(T-I)(x)$} $\|Tx^{k}-x^{k}\|\to0$,
\\then $x^{k}$ weakly converges to some $x^{*}\in\text{Fix}(T)$
\footnote{Weak convergence is the same as regular convergence in $\RR^N$, but differs in a general Hilbert space.}.

\end{prop}

\Cref{prop:Convergence-of-fixed-point-algorithms} is the basis of our convergence proofs in this paper, as well the proof of convergence of KM iteration. Toward applying  \Cref{prop:Convergence-of-fixed-point-algorithms}, we study the following:
\begin{enumerate}
\item \textbf{Building a Lyapunov function:\footnote{Technically this is not a Lyapunov function, but it resembles one.}} It turns out to be more natural
to look at the Lyapunov function:
\begin{align}
\underbrace{\xi^{k}}_{\text{Total error}} &= \underbrace{\n{ x^{k}-x^{*}} ^{2}}_{\text{Classical error}}+\underbrace{\frac{1}{m}\sum_{i=1}^{\infty}c_{i}\n{ x^{k+1-i}-x^{k-i}} ^{2}}_{\text{Asynchronicity error}}
\end{align}
rather than the classical error $\n{ x^{k}-x^{*}} ^{2}$ alone. Here, we let $x^n=x^0$ form $n<0$. We cannot ensure that $\EE\sp{\n{ x^{k+1}-x^{*}} ^{2}}<\n{ x^{k}-x^{*}} ^{2}$ due to asynchronicity, and generally some kind of monotonicity result is needed to prove convergence. However adding what we might call the \textbf{asynchronicity error}, we regain this monotonicity of expectation, which leads to a viable proof.

\item \textbf{Martingale convergence theory}: This allows us to prove \emph{norm
convergence} and \emph{FPR strong convergence} using results on the above Lyapunov function, which will complete the proof. Martingale theory is what allowed the authors in \cite{PengXuYanYin2015_arock} to prove that $x^k$ converges to a solution for minimization of a convex function with Lipschitz gradient, and not just that the function value converged to the optimal value.
\end{enumerate}

\subsection{Preliminary results} \label{sec:Preliminary-results}
Recall that stochastic unbounded delays are analyzed under \Cref{asmp:Block-sequence,asmp:Stochastic-delays}. Define, for $k=0,1,\ldots,$ the filtration
\begin{align}
\cF^{k} & = \sigma\p{x^{0},x^{1},\ldots,x^{k},\vec{j}\p{0},\vec{j}\p{1},\ldots,\vec{j}\p{k}},
\end{align}
which represents the history of iterates and delays up to the present step $k$.
Let $x^*$ be any solution, and set $x^*=0$ with no loss in generality, to make some notation more compact. This can be achieved by translating the origin of the coordinate system to $x^*$. Thence, $\n{ x^{k}} $ is the distance from the solution\footnote{We will use an abuse of notation in this paper. We equate $S_i(x)\in\HH_i$ (the components of $S(x)$ in the $i$th block) and $(0,\ldots,0,S_i(x),0,\ldots,0)\in\HH_1\times\ldots\times\HH_m$ (the projection of $S(x)$ to the $i$'th subspace). Hence we can write the ARock iteration more compactly as $x^{k+1}=x^{k}-\eta^{k}S_{i\p{k}}\hx^{k}$.}:
\begin{align}
\EE\sp{\n{ x^{k+1}} ^{2}|\cF^{k}} & = \EE\sp{\n{ x^{k}-\eta^{k}S_{i\p{k}}\hx^{k}} ^{2}|\cF^{k}}\\
 & = \n{ x^{k}} ^{2}+\EE\sp{-2\eta^{k}\dotp{ x^{k},S_{i\p{k}}\hx^{k}} +\p{\eta^{k}}^{2}\n{ S_{i\p{k}}\hx^{k}} ^{2}|\cF^{k}},
\end{align}
where the expectation is taken over only the block index $i\p{k}$ only. Since the step size $\eta^{k}$ is chosen independently of $i\p{k}$ and \Cref{asmp:Block-sequence,asmp:Stochastic-delays} hold, we obtain
\begin{align}\label{eq:Starting-point}
\EE\sp{\n{ x^{k+1}} ^{2}|\cF^{k}} & = \n{ x^{k}} ^{2}\underbrace{-2\frac{\eta^{k}}{m}\dotp{ x^{k},S\hx^{k}} }_{\text{cross term}}+\frac{\p{\eta^{k}}^{2}}{m}\n{ S\hx^{k}}^{2}.
\end{align}

\subsubsection{A fundamental inequality}
We start with a fundamental inequality, which is the starting point for analyzing convergence.

\begin{prop}[Fundamental inequality]\label{prop:Fundamental-inequality}
Under  \Cref{asmp:Block-sequence,asmp:Stochastic-delays}, for $j\p{k}$ defined in \eqref{eq:def:current-delay}, and an arbitrary sequence $\eps_{1},\,\eps_{2},\ldots\in(0,\infty)$, the ARock iterates obey the following inequality:
\begin{align}\label{eq:Fundamental-inequality}
\begin{aligned}
\EE\sp{\n{ x^{k+1}-x^{*}} ^{2}\big|\cF^{k}} &\leq \n{ x^{k}-x^{*}} ^{2}+\frac{1}{m}\sum_{i=1}^{j\p{k}}\eps_{i}\n{ x^{k+1-i}-x^{k-i}} ^{2}
\quad - \frac{\eta^{k}}{m}\n{ S\hx^{k}} ^{2}\p{1-\eta^{k}\p{1+\sum_{i=1}^{j\p{k}}\frac{1}{\eps_{i}}}}.
\end{aligned}
\end{align}
\end{prop}
The $\eps_{i}$ sequence  is tunable. In \cite{PengXuYanYin2015_arock}, they are set to a constant value. However we eventually set them so
that $1/\eps_{i}$ is summable, which is fundamental to the convergence proof for unbounded delays.

\begin{proof}
Let us start with the cross term in \eqref{eq:Starting-point}. Since $T$ is nonexpansive,  $\frac{1}{2}S$ is firmly nonexpansive (FNE)\footnote{A firmly nonexpansive (FNE) operator $Q:\HH\to\HH$ is an operator that can be written as $Q=\frac{1}{2}I+\frac{1}{2}R$, where $R$ is nonexpansive. Equivalently, FNE operators satisfy $\dotp{Qy-Qx,y-x} \geq \n{Qy-Qx}^2,~ \forall x,y\in\HH$.}. Hence,
\begin{align*}
-\frac{2\eta^{k}}{m}\dotp{ S\hx^{k},x^{k}} & =  -\frac{2\eta^{k}}{m}\p{\dotp{S\hx^{k},\hx^{k}} + \dotp{ S\hx^{k},x^{k}-\hx^{k}}}\\
& =  -\frac{2\eta^{k}}{m}\p{2\dotp{\frac{1}{2}S\hx^{k},\hx^{k}}+\dotp{ S\hx^{k},x^{k}-\hx^{k}}}\\
\p{\frac{1}{2}S~\text{is FNE}} & \leq -\frac{2\eta^{k}}{m}\p{2\n{ \frac{1}{2}S\hx^{k}} ^{2} + \dotp{ S\hx^{k},x^{k}-\hx^{k}}}\\
& =  -\frac{\eta^{k}}{m}\n{ S\hx^{k}} ^{2} - \frac{2\eta^{k}}{m}\dotp{ S\hx^{k},x^{k}-\hx^{k}}\\
\text{(break into coordinate blocks)} & =  \sum_{l=1}^{m}\p{-\frac{\eta^{k}}{m}\n{ S_{l}\hx^{k}} ^{2}-\frac{2\eta^{k}}{m} \dotp{S_{l}\hx^k, x_{l}^{k}-x_l^{k-j\p{k,l}}}}.
\end{align*}
Take block $l$. We turn the inner product into a telescoping sum:
\begin{align*}
&\quad -\frac{\eta^{k}}{m}\n{ S_{l}\hx^{k}} ^{2}-\frac{2\eta^{k}}{m}\dotp{ S_{l}\hx^k,x_{l}^{k}-x_l^{k-j\p{k,l}}}\\
&= -\frac{\eta^{k}}{m}\n{ S_{l}\hx^{k}} ^{2}-\frac{2\eta^{k}}{m}\p{\sum_{i=1}^{j\p{k,l}}\dotp{ S_{l}\hx^{k},x_{l}^{k+1-i}-x_{l}^{k-i}}}\\
\text{(Cauchy-Schwarz)}  &\leq -\frac{\eta^{k}}{m}\n{ S_{l}\hx^{k}} ^{2}+\frac{2\eta^{k}}{m}\p{\sum_{i=1}^{j\p{k,l}}\frac{1}{2}\p{\n{ S_{l}\hx^{k}} ^{2}\frac{\eta^{k}}{\eps_{i}}+\frac{\eps_{i}}{\eta^{k}}\n{ x_{l}^{k+1-i}-x_{l}^{k-i}} ^{2}}}\\
&\leq -\frac{\eta^{k}}{m}\n{ S_{l}\hx^{k}} ^{2}+\frac{\eta^{k}}{m}\p{\sum_{i=1}^{j\p{k}}\p{\n{ S_{l}\hx^{k}} ^{2}\frac{\eta^{k}}{\eps_{i}}+\frac{\eps_{i}}{\eta^{k}}\n{ x_{l}^{k+1-i}-x_{l}^{k-i}} ^{2}}}\\
&= \quad \frac{\eta^{k}}{m}\n{ S_{l}\hx^{k}} ^{2}\p{\eta^{k}\p{\sum_{i=1}^{j\p{k}}\frac{1}{\eps_{i}}}-1}+\frac{1}{m}\sum_{i=1}^{j\p{k}}\eps_{i}\n{ x_{l}^{k+1-i}-x_{l}^{k-i}} ^{2}.
\end{align*}
Adding all the components back together, we have:
\begin{align*}
-2\frac{\eta^{k}}{m}\dotp{ x^{k},S\hx^{k}} +\frac{\p{\eta^{k}}^{2}}{m}\n{ S\hx^{k}} ^{2} &\leq \frac{\eta^{k}}{m}\n{ S\hx^{k}} ^{2}\p{\eta^{k}\p{1+\p{\sum_{i=1}^{j\p{k}}\frac{1}{\eps_{i}}}}-1}+\frac{1}{m}\sum_{i=1}^{j\p{k}}\eps_{i}\n{ x^{k+1-i}-x^{k-i}} ^{2}.
\end{align*}
Hence the proposition follows by adding $\n{x^k}^2$ to each side, and using \eqref{eq:Starting-point}.
\end{proof}

\subsection{Building a Lyapunov function}
In this section we demonstrate how to construct a Lyapunov function from \eqref{eq:Fundamental-inequality} to prove convergence.

When calculating $\EE\sp{\n{ x^{k+1}-x^*} ^{2}\big|\cF^{k}}$, notice that we obtained some difference terms of the form $\n{ x^{k+1-i}-x^{k-i}} ^{2}$. These difference terms are not easy to compare with $\n{ x^{k}-x^*} ^{2}$, and hence we cannot immediately say anything about the growth of the error. Instead of just considering $\n{ x^{k}-x^*} ^{2}$, we consider a Lyapunov function $\xi^{k}$ defined as follows:

\begin{defn}[The Lyapunov function] \label{def:Lyapunov-definition}
Let $x^0, x^1, x^2, \ldots$ be a sequence of points in $\HH$, and let $c_1, c_2, c_3, \ldots$ be a sequence of parameters in $[0,\infty)$. Set $x^n=x^0$ for all $n<0$. We define the Lyapunov function:
\begin{align}\label{eq:def:Lyapunov}
\xi^{k} &= \n{ x^{k}-x^{*}} ^{2}+\frac{1}{m}\sum_{i=1}^{\infty}c_{i}\n{ x^{k+1-i}-x^{k-i}} ^{2}.
\end{align}
\end{defn}
This is simply a linear combination of all the terms found when calculating the expectation of the original error. It is similar to, but different from, that used in \cite{PengXuYanYin2015_arock}. When we calculate $\EE\sp{\xi^{k+1}\big|\cF^{k}}$, we hope to only have terms similar to the terms found in $\xi^{k}$: that is, only terms like $\n{x^k-x^*}^2$ and $\n{x^{k+1-i}-x^{k-i}}$, and not some third species of terms. If this is the case, then we may carefully chose the coefficients $c_{1},c_{2},\ldots$ so that we may compare $\xi^{k}$ and $\xi^{k+1}$ in a meaningful way. Information about how \textit{some kind of error} grows is essential to convergence proofs.

\subsubsection{Analysis of the Lyapunov function}
We now analyze the conditional expectation of the Lyapunov function defined in \eqref{eq:def:Lyapunov}.

\begin{lem}[Branch point lemma] \label{lem:Branch-point-lemma}
Take arbitrary $\eps_{1},\,\eps_{2},\ldots\in(0,\infty)$. Under  \Cref{asmp:Block-sequence,asmp:Stochastic-delays},
 the ARock iterates and $\xi^k$ defined in \eqref{eq:def:Lyapunov} satisfy the following inequality:
\begin{align}
\begin{aligned}
\EE\sp{\xi^{k+1}\big|\cF^{k}} & \leq \n{ x^{k}} ^{2}+\frac{1}{m}\p{\sum_{i=1}^{j\p{k}}\eps_{i}\n{ x^{k+1-i}-x^{k-i}} ^{2}+\sum_{i=1}^{\infty}c_{i+1}\n{ x^{k+1-i}-x^{k-i}} ^{2}}\\
&\quad -\frac{\eta^{k}}{m}\n{ Sx^{k-\vec{j}\p{k}}} ^{2}\p{1-\eta^{k}\p{1+\frac{c_{1}}{m}+\sum_{i=1}^{j\p{k}}\frac{1}{\eps_{i}}}}.
\end{aligned}
\end{align}
\end{lem}

\begin{proof}
Calculate the expectation:
\begin{align}
\EE\sp{\xi^{k+1}\big|\cF^{k}} & = \EE\sp{\n{ x^{k+1}} ^{2}\big|\cF^{k}}+\frac{c_{1}}{m}\EE\sp{\n{ x^{k+1}-x^{k}} ^{2}\big|\cF^{k}}+\frac{1}{m}\sum_{i=1}^{\infty}c_{i+1}\n{ x^{k+1-i}-x^{k-i}} ^{2}.
\end{align}
The second term yields (by the definition of ARock iteration \eqref{eq:def:ARock}, and taking expectation over $i(k)$)
\begin{align}
\EE\sp{\n{ x^{k+1}-x^{k}} ^{2}\big|\cF^{k}} &= \frac{\p{\eta^{k}}^{2}}{m}\n{ S(x^{k-j\p{k}})} ^{2}.\label{eq:Difference-term}
\end{align}
Then,
\begin{align*}
\EE\sp{\xi^{k+1}\big|\cF^{k}} & = \underbrace{\EE\sp{\n{ x^{k+1}} ^{2}\big|\cF^{k}}}_{A}+\underbrace{\frac{c_{1}}{m}\EE\sp{\n{ x^{k+1}-x^{k}} ^{2}\big|\cF^{k}}}_{B}+\underbrace{\frac{1}{m}\sum_{i=1}^{\infty}c_{i+1}\n{ x^{k+1-i}-x^{k-i}} ^{2}}_{C}\\
& \leq \underbrace{\n{ x^{k}} ^{2}+\frac{\eta^{k}}{m}\n{ Sx^{k-\vec{j}\p{k}}} ^{2}\p{\eta^{k}\p{1+\sum_{i=1}^{j\p{k}}\frac{1}{\eps_{i}}}-1}+\frac{1}{m}\sum_{i=1}^{j}\eps_{i}\n{ x^{k+1-i}-x^{k-i}} ^{2}}_{A}\text{ (by \eqref{eq:Fundamental-inequality})}\\
  & \quad +\underbrace{\frac{c_{1}}{m}\p{\frac{\p{\eta^{k}}^{2}}{m}\n{ Sx^{k-\vec{j}\p{k}}} ^{2}}}_{B}\text{ (by \eqref{eq:Difference-term})}\\
  & \quad \underbrace{+\frac{1}{m}\sum_{i=1}^{\infty}c_{i+1}\n{ x^{k+1-i}-x^{k-i}} ^{2}}_{C}\\
 &= \n{ x^{k}} ^{2}+\frac{1}{m}\p{\sum_{i=1}^{j\p{k}}\eps_{i}\n{ x^{k+1-i}-x^{k-i}} ^{2}+\sum_{i=1}^{\infty}c_{i+1}\n{ x^{k+1-i}-x^{k-i}} ^{2}}\\
  & \quad -\frac{\eta^{k}}{m}\n{ Sx^{k-\vec{j}\p{k}}} ^{2}\p{1-\eta^{k}\p{1+\frac{c_{1}}{m}+\sum_{i=1}^{j\p{k}}\frac{1}{\eps_{i}}}}.
\end{align*}
\end{proof}
Define
\begin{align}
\cG^{k} = \sigma\p{x^{0},x^{1},\ldots,x^{k}},
\end{align}
which represents the history of iterates $x^0,x^1,x^2,\ldots$.
In the proposition below, we derive the natural choice of parameters of the Lyapunov function that allow a meaningful comparison between $\EE\sp{\xi^{k+1}\big|\
G^{k}}$ and $\xi^{k}$. With this choice, we obtain
\begin{align*}
\EE\sp{\xi^{k+1}\big|\cG^{k}}\leq\xi^{k}-\p{\text{descent terms}},
\end{align*}
which strongly resembles norm convergence: one of the convergence conditions in \Cref{prop:Convergence-of-fixed-point-algorithms}.

We first make some assumptions on the parameters. The necessity of these assumptions will become clear in the proof of \Cref{lem:Descent-lemma-stochastic}.
\begin{asmp}[Coefficient summability conditions] \label{asmp:Coefficient-formula-stochastic}
Let $\eps_{1}, \eps_{2},\ldots\in(0,\infty)$ and let $c_{i}=\sum_{l=i}^{\infty}\eps_{l}P_{l}$. These sequences also satisfy the summability conditions:
\begin{align}
\sum_{i=1}^\infty \frac{1}{\eps_i} <& \infty,\label{eq:asmp:One-on-epsilon-l1}\\
\sum_{i=1}^\infty c_i <& \infty.\label{eq:asmp:Coefficients-summable-stochastic}
\end{align}
\end{asmp}

\begin{lem}[Descent lemma for stochastic delays] \label{lem:Descent-lemma-stochastic}
Consider the Lyapunov function $\xi^{k}$ defined in \eqref{eq:def:Lyapunov}. Let \Cref{asmp:Block-sequence,asmp:Stochastic-delays,asmp:Coefficient-formula-stochastic} hold. Let $h=\p{1+\frac{c_{1}}{m}+\n{ \frac{1}{\eps_{i}}} _{\ell^{1}}}^{-1}$. Then, ARock yields the following inequality for step size $\eta^k$:
\begin{align*}
\EE\sp{\xi^{k+1}\big|\cG^{k}} &\leq \xi^{k}-\p{1-\eta^{k}/h}\frac{\eta^{k}}{m}\sum_{\vec{j}\in\NN^{m}}p\p{\vec{j}}\n{ Sx^{k-\vec{j}}} ^{2}.
\end{align*}
\end{lem}

\begin{proof}
From \Cref{lem:Branch-point-lemma} and \eqref{eq:asmp:One-on-epsilon-l1}, we have:
\begin{align}
\begin{aligned}
\EE\sp{\xi^{k+1}\big|\cF^{k}} &\leq \n{ x^{k}} ^{2}+\frac{1}{m}\p{\sum_{i=1}^{j\p{k}}\eps_{i}\n{ x^{k+1-i}-x^{k-i}} ^{2}+\sum_{i=1}^{\infty}c_{i+1}\n{ x^{k+1-i}-x^{k-i}} ^{2}}\\
&\quad -\frac{\eta^{k}}{m}\n{ Sx^{k-\vec{j}\p{k}}} ^{2}\p{1-\eta^{k}\underbrace{\p{1+\frac{c_{1}}{m}+\n{ \frac{1}{\eps_{i}}} _{\ell^{1}}}}_{1/h}}.
\end{aligned}
\end{align}
Let $p_{j}=\PP\sp{j\p{k}=j}$. Now take expectations over delays (via taking expectation with respect to $\cG^{k}$ instead of $\cF^{k}$).
\begin{align*}
\EE\sp{\xi^{k+1}\big|\cG^{k}} & \leq \n{ x^{k}} ^{2}+\frac{1}{m}\p{\sum_{j=1}^{\infty}p_{j}\sum_{i=1}^{j}\eps_{i}\n{ x^{k+1-i}-x^{k-i}} ^{2}+\sum_{i=1}^{\infty}c_{i+1}\n{ x^{k+1-i}-x^{k-i}} ^{2}}\\
&\quad -\p{1-\eta^{k}/h}\frac{\eta^{k}}{m}\sum_{\vec{j}\in\NN^{m}}p\p{\vec{j}}\n{ Sx^{k-\vec{j}}} ^{2}\\
&= \n{ x^{k}} ^{2}+\frac{1}{m}\p{\sum_{i=1}^{\infty}\p{\sum_{j=i}^{\infty}p_{j}}\eps_{i}\n{ x^{k+1-i}-x^{k-i}} ^{2}+\sum_{i=1}^{\infty}c_{i+1}\n{ x^{k+1-i}-x^{k-i}} ^{2}}\\
& \quad -\p{1-\eta^{k}/h}\frac{\eta^{k}}{m}\sum_{\vec{j}\in\NN^{m}}p\p{\vec{j}}\n{ Sx^{k-\vec{j}}} ^{2}\\
& = \n{ x^{k}} ^{2}+\frac{1}{m}\p{\sum_{i=1}^{\infty}\p{\eps_{i}P_{i}+c_{i+1}}\n{ x^{k+1-i}-x^{k-i}} ^{2}}-\p{1-\eta^{k}/h}\frac{\eta^{k}}{m}\sum_{\vec{j}\in\NN^{m}}p\p{\vec{j}}\n{ Sx^{k-\vec{j}}} ^{2}.
\end{align*}
Let $\eta^{k}\leq h$ to eliminate the last term. Ideally $\EE\sp{\xi^{k+1}\big|\cG^{k}}\leq\xi^{k}$, which can be achieved with:
\begin{align*}
\n{ x^{k}} ^{2}+\frac{1}{m}\p{\sum_{i=1}^{\infty}\p{\eps_{i}P_{i}+c_{i+1}}\n{ x^{k+1-i}-x^{k-i}} ^{2}} & \leq \n{ x^{k}} ^{2}+\frac{1}{m}\sum_{i=1}^{\infty}c_{i}\n{ x^{k+1-i}-x^{k-i}} ^{2}.
\end{align*}
The obvious choice of coefficients is then given by $c_{i+1} + P_i\eps_i = c_i$. However this doesn't uniquely determine the coefficients. We assume that $c_i\to 0$ as $i$ goes to $\infty$ to ensure that any bounded sequence has a corresponding Lyapunov function that is finite. Hence:
\begin{align}
c_{i} &= \sum_{l=i}^{\infty}\eps_{i}P_{i}.\nonumber
\end{align}
This recovers the coefficient formula from \Cref{asmp:Coefficient-formula-stochastic}. With this choice of coefficients, we have our result.
\end{proof}

\subsection{Convergence proof} \label{sec:sub:Convergence-proof-stochastic}
Now that we have built a Lyapunov function and obtained \Cref{lem:Descent-lemma-stochastic}, we can prove convergence.

\begin{lem}\label{lem:Sum-Sxhat-stochastic}
Let \Cref{asmp:Block-sequence,asmp:Stochastic-delays,asmp:Coefficient-formula-stochastic} hold. Use step size $\eta^{k}=ch$ for some arbitrary fixed $c\in\p{0,1}$,
and $h$ given in \Cref{lem:Descent-lemma-stochastic}. Then with probability $1$, $\xi^{k}$ converges, and in addition,
\begin{align}\label{eq:lem:Sum-Sxhat-stochastic}
\sum_{k=0}^{\infty}\sum_{\vec{j}\in\NN^{m}}p\p{\vec{j}}\n{ Sx^{k-\vec{j}}} ^{2} &< \infty.
\end{align}
\end{lem}

The proof of this lemma relies on the following:

\begin{thm}[Supermartingale convergence theorem \cite{CombettesPesquet2015_stochastic}]\label{thm:Supermartingale-convergence-theorem}
Let $\alpha^{k}$, $\theta^{k}$ and $\gamma^{k}$ be positive sequences
adapted to $\cF^{k}$, and let $\gamma^{k}$ be summable with probably 1. If
\begin{align*}
\EE\sp{\alpha^{k+1}|\cF^{k}}+\theta^{k} & \leq \alpha^{k}+\gamma^{k},
\end{align*}
then with probability 1, $\alpha^{k}$ converges to a $[0,\infty)$-valued random variable,
and $\sum_{k=1}^{\infty}\theta^{k}<\infty$.
\end{thm}
We now prove \Cref{lem:Sum-Sxhat-stochastic}.
\begin{proof}
Apply \Cref{thm:Supermartingale-convergence-theorem} with
$\alpha^{k}=\xi^{k}$, $\gamma^k=0$, and $\theta^{k}=\p{1-\eta^{k}/h}\frac{\eta^{k}}{m}\sum_{\vec{j}\in\NN^{m}}p\p{\vec{j}}\n{ Sx^{k-\vec{j}}} ^{2}$.
We immediately obtain our result by noting that $\p{1-\eta^{k}/h}\frac{\eta^{k}}{m}$
is a constant.
\end{proof}

\subsubsection{Norm convergence}
Now is the point where the ``evenly old'' assumption about the delays made in \Cref{asmp:Stochastic-delays} becomes important, and it is hard to see a way to weaken it. First a lemma on convolutions is necessary.

\begin{lem}[Convolution lemma (\cite{ArendtBattyHieberNeubrander2011_vectorvalued}, Proposition 1.3.2)] \label{lem:Convolution-lemma}
Define the convolution of sequences $a=(\ldots, a_{-2}, a_{-1}, a_0, a_1, a_2,\ldots)$ and $b=(\ldots, b_{-2}, b_{-1}, b_0, b_1, b_2,\ldots)$ as the sequence defined by the formula\footnote{The convolution is not always well-defined, because the sum may not be convergent for all $k$. However in this lemma, it is well-defined.}:
\begin{align}
\p{a*b}(k) &= \sum_{i=-\infty}^{\infty}a_{i}b_{k-i}.
\end{align}
Let $a_{i}$ be in $\ell^{1}$, and let $b$ be bounded with $b_{i}\to0$ as $i\to\infty$. Then the convolution $(a*b)\p{k}\to0$ as $k\to\infty$.
\end{lem}

\begin{prop}[Norm convergence] \label{prop:Norm-convergence-stochastic}
Let \Cref{asmp:Block-sequence,asmp:Stochastic-delays,asmp:Coefficient-formula-stochastic} hold. Then
with probability $1$, $\n{ x^{k}-x^{*}} $ converges for all $x^{*}\in\normalfont{\text{Fix}}\p{T}$.

\end{prop}

\begin{proof}
We first prove that with probability $1$, $\frac{1}{m}\sum_{i=1}^{\infty}c_{i}\n{ x^{k+1-i}-x^{k-i}} ^{2}\to0$.

\textbf{1. ${P_{l}}$ is summable.}
Since the sequence $\frac{1}{\eps_{i}}$ is summable, $\eps_{i}\to\infty$, and thus $\inf_{i\in\NN} \eps_i>0$.
Hence
\begin{align*}
\sum_{l=1}^{\infty}P_{l} &\leq \frac{1}{\inf_{i\in\NN} \eps_i}\sum_{l=1}^{\infty}\eps_{l}P_{l}=\frac{1}{\inf_{i\in\NN} \eps_i}c_1<\infty
\end{align*}
\textbf{2. ${k-j\p{k}\to\infty}$.} (That is, the components of iterate $x^k$ are used only a finite number of times).
\begin{align*}
\PP\sp{k-j\p{k}\leq k_{0}} &= P_{k-k_{0}}\\
\sum_{k=k_0}^\infty \PP\sp{k-j\p{k}\leq k_{0}} &= \sum_{k=k_0}^\infty P_{k-k_{0}} < \infty
\end{align*}
Therefore by the Borel-Cantelli lemma, $k-j\p{k}\leq k_{0}$ happens only a finite number of times with probability $1$. Hence with probability $1$, this is true for all $k_{0}\in\NN$, which implies that $k-j\p{k}\to\infty$.

\textbf{3. $Sx^{k+\vec{t}}\to0$ for all delay feasible ``patterns'' $\vec{t}$.}
We assume without loss in generality that none of the delay vectors attained ($\vec{j}\p{0},\vec{j}\p{1},\ldots$) has probability $0$ (since this occurs with probability 1). Let $\vec{t}\p{k}\triangleq j\p{k}\p{1,\ldots,1}-\vec{j}\p{k}$. $j\p{k}$ is the age of the oldest block in $x^{k-\vec{j}\p{k}}$, whereas $\vec{t}\p{k}\in\cp{ 0,1,\ldots,B} ^{m}$ represent the ``pattern'' of the rest of the delay. We call a vector $\vec{t}\in\cp{ 0,1,\ldots,B} ^{m}$ \textbf{feasible} if it occurs with nonzero probability. Take \eqref{eq:lem:Sum-Sxhat-stochastic}, and group the sum into feasible patterns and we obtain:
\begin{align}
\sum_{k=0}^{\infty}\n{ Sx^{k+\vec{t}}} ^{2} < & \infty,\nonumber \\
\implies\n{ Sx^{k+\vec{t}}} \to & 0,\label{eq:Sx^(k+t)-goes-to-0}
\end{align}
for each feasible $\vec{t}$.

\textbf{4. Delayed fixed-point residual} ${\n{ Sx^{k-\vec{j}\p{k}}} \to0}$. Observe that
\begin{align*}
\n{ Sx^{k-\vec{j}\p{k}}}  &= \n{ Sx^{\p{k-j\p{k}}+\vec{t}\p{k}}}.
\end{align*}
Let $A\p{k,\vec{t}}= \n{ Sx^{\p{k-j\p{k}}+\vec{t}}}$ (this is a family of sequences indexed by $\vec{t}$). By equation \eqref{eq:Sx^(k+t)-goes-to-0}, and the fact that $k-j\p{k}\to\infty$, we have $A(k,\vec{t})\to 0$ for any \textit{fixed} $\vec{t}$. Notice that $\n{ Sx^{k-\vec{j}\p{k}}}=A(k,\vec{t}(k))$. At every step, $A(k,\vec{t}(k))$ selects one from a finite family of sequences, all of which converge to $0$. Since there are only a finite number of these sequences, $A(k,\vec{t}(k))\to 0$ and hence $\n{Sx^{k-\vec{j}\p{k}}}\to0$\footnote{If you select from an \textit{infinite} number of sequences converging to $0$, this may not be true. E.g. consider $B(k,i)=\delta_{k-i}$, where $\delta_0=1$ and $\delta_l=0$ for all $l\neq0$. For fixed $i$, $B(k,i)\to 0$. However $B(k,k)=1$ for all $k$, and hence never converges to $0$.}.

\textbf{5. Difference sum converges to $0$.}
\begin{align*}
&\quad\frac{1}{m}\sum_{i=1}^{\infty}c_{i}\n{ x^{k+1-i}-x^{k-i}} ^{2}\\
&\leq \frac{c^2h^{2}}{m}\sum_{i=1}^{\infty}c_{i}\n{ Sx^{\p{k-i}-\vec{j}\p{k-i}}}^{2}\\
&=\frac{c^2h^{2}}{m}\p{\p{0,\ldots,0,c_1,c_2,\ldots}*\p{\ldots,\n{ Sx^{\p{i-1}-\vec{j}\p{i-1}}}^{2},\n{ Sx^{\p{i}-\vec{j}\p{i}}}^{2},\n{ Sx^{\p{i+1}-\vec{j}\p{i+1}}}^{2},\ldots}}(k)
\end{align*}
This expression is the convolution of an $\ell^1$ sequence (\Cref{asmp:Coefficient-formula-stochastic}), and a bounded sequence that converges to $0$ as $i\to\infty$ (by part 4 of this proof) respectively. Therefore by \Cref{lem:Convolution-lemma}, $\frac{1}{m}\sum_{i=1}^{\infty}c_{i}\n{ x^{k+1-i}-x^{k-i}} ^{2}\to0$.

\textbf{6. Norm convergence.} Because $\xi^{k}$ converges a.s. and
$\frac{1}{m}\sum_{i=1}^{\infty}c_{i}\n{ x^{k+1-i}-x^{k-i}} ^{2}\to0$
a.s., we have that for any particular $x^{*}$, $\n{ x^{k}-x^{*}} $
converges with probability $1$. Because the space is \textit{separable}, this implies that with probability $1$, $\n{ x^{k}-x^{*}}$ converges for \textbf{all} $x^{*}\in\text{Fix}\p{T}$, which is subtly different (See \cite{CombettesPesquet2015_stochastic}, Proposition 2.3 (iii) for a proof of this fact.).
\end{proof}

\subsubsection{Fixed-point-residual strong convergence}

\begin{prop}[FPR strong convergence]
\label{prop:FPRSC-stochastic}
Under the conditions of \Cref{prop:Norm-convergence-stochastic}, $\n{ Sx^{k}} \to0$ with probability 1.

\end{prop}

\begin{proof}
From equation \eqref{eq:lem:Sum-Sxhat-stochastic}, we have that
$\n{ Sx^{k+\vec{t}}} \to0$ for some feasible $\vec{t}$ (clearly there must be at least one feasible $\vec{t}$). Recall that $m$ is the number of blocks, and $B$ is the maximum
difference in age between blocks. We have
\begin{align*}
\n{ Sx^{k}} &\leq \n{ Sx^{k+\vec{t}}-Sx^{k}} +\n{ Sx^{k+\vec{t}}} \\
&\leq 2\n{ x^{k+\vec{t}}-x^{k}} +\n{ Sx^{k+\vec{t}}} \\
\text{(triangle inequality)}  &\leq 2\sum_{i=1}^{m}\n{ x_{i}^{k+t_{i}}-x_{i}^{k}} +\n{ Sx^{k+\vec{t}}} \\
&\leq 2\sum_{i=1}^{m}\sum_{l=1}^{t_i}\n{ x_{i}^{k+l}-x_{i}^{k-1+l}} +\n{ Sx^{k+\vec{t}}} \\
\text{(since $\vec{t}\in\{0,1,\ldots,B\}^m$)} &\leq 2m\sum_{l=1}^{B}\n{ x_{i}^{k+l}-x_{i}^{k-1+l}} +\n{ Sx^{k+\vec{t}}} \to 0,
\end{align*}
since $\n{ x^{k+1}-x^{k}} \to 0$ and $\n{ Sx^{k+\vec{t}}} \to 0$ (from parts 5 and 3 of the proof of \Cref{prop:Norm-convergence-stochastic} respectively).
\end{proof}

\subsubsection{Proof of \Cref{thm:Convergence-stochastic-delays-full}} \label{sec:sub:sub:Proof-of-theorem-stochastic}

\begin{proof}
Norm convergence is proven in \Cref{prop:Norm-convergence-stochastic}. The FPR strong convergence criterion is proven in \Cref{prop:FPRSC-stochastic}. Having satisfied the conditions of \Cref{prop:Convergence-of-fixed-point-algorithms}, we conclude that the sequence of ARock iterates converges to a solution with probability $1$. Hence we have proven \Cref{thm:Convergence-stochastic-delays-full}.
\end{proof}

\subsection{Parameter choice} \label{sec:Parameter-choice-stochastic}
Choosing different parameters $\eps_{1},\eps_{2},\ldots$ lead to different convergence results. We featured two possibilities in \Cref{thm:Convergence-stochastic-delays-intro} (though there are obviously others). We need both $\frac{1}{\eps_{i}}\in\ell^{1}$ and $\sum_{l=1}^{\infty}c_{l}=\sum_{l=1}^{\infty}\eps_{l}P_{l}l<\infty$ for convergence under step size $\eta^{k} =ch =c\p{1+\frac{1}{m}\sum_{l=1}^{\infty}\eps_{l}P_{l}+\n{ \frac{1}{\eps_{i}}}_{\ell^{1}}}^{-1}$.
\begin{enumerate}
\item If we wish to have the weakest restriction on our distribution of
delays, let $\eps_{l}=m^{-1/2}P_{l}^{-1/2}l^{-1/2}$. This leads
to the convergence condition $\sum_{l=1}^{\infty}P_{l}^{1/2}l^{1/2}<\infty$
for step size $\eta^{k} = c\p{1+\frac{1}{\sqrt{m}}\sum_{l=1}^{\infty}P_{l}^{1/2}\p{l^{1/2}+l^{-1/2}}}^{-1}$.
\item If we wish to have the largest allowable step size (at the expense
of a strong condition on the delay distribution), let $\eps_{l}=m^{-1/2}P_{l}^{-1/2}$.
This leads to the convergence condition $\sum_{l=1}^{\infty}P_{l}^{1/2}l^{1}<\infty$
for step size $\eta^{k}=c\p{1+\frac{2}{\sqrt{m}}\sum_{l=1}^{\infty}P_{l}^{1/2}}^{-1}$.
\end{enumerate}

\subsection{General strategy} \label{sec:General-strategy}
The general strategy for building Lyapunov functions is as follows. This has wide applicability in optimization, and not just asynchronous algorithms.

\textbf{General Strategy:}

\begin{rem}[General Strategy]
\begin{enumerate}
\item Let $\xi^{k}$ initially be the classical error $\n{ x^{k+1}} ^{2}$ (or $f(x^k)-f(x^*)$, or similar). We will adaptively change $\xi$, until we have a useful Lyapunov function. Calculate the expectation of the classical error $\EE\sp{\n{ x^{k+1}} ^{2}\big|\cF^{k}}$ and take inequalities (See \Cref{sec:Preliminary-results} where we obtained \Cref{prop:Fundamental-inequality}, the fundamental inequality.).
\item If this produces residual terms (in our case $\n{ x^{k+1-i}-x^{k-i}} ^{2}$) that we cannot eliminate, add a general linear combination of these terms to $\xi^{k}$. In this case, we add $\frac{1}{m}\sum_{i=1}^{\infty}c_{i}\n{ x^{k+1-i}-x^{k-i}} ^{2}$ to obtain the Lyapunov function in \Cref{def:Lyapunov-definition}.
\item Repeat steps 1 and 2 until we gain ``closure''. I.e. The \textit{positive terms} in the inequality for $\EE\sp{\xi^{k+1}\big|\cF^{k}}$ are the same as the terms found in $\xi^{k}$ (In our case, we only needed to do this once.).
\item\textit{Negative terms} are not problematic because they serve to decrease the expectation of $\xi^{k+1}$. They should not be eliminated because they can give useful information. In our case,
\begin{align}
-\frac{\eta^{k}}{m}\n{ Sx^{k-\vec{j}\p{k}}} ^{2}\p{1-\eta^{k}\p{1+\frac{c_{1}}{m}+\sum_{i=1}^{j\p{k}}\frac{1}{\eps_{i}}}}
\end{align}
was a negative term (see \Cref{lem:Descent-lemma-stochastic}).  This term was critical in the proof of the norm convergence and FPR strong convergence criterion in \Cref{sec:sub:Convergence-proof-stochastic}. See \Cref{lem:Sum-Sxhat-stochastic}, \Cref{prop:Norm-convergence-stochastic,prop:FPRSC-stochastic}.).
\item Vary the coefficients of the Lyapunov function to enable a useful comparison between $\EE\sp{\xi^{k+1}\big|\cF^{k}}$ and $\xi^{k}$ (See \Cref{lem:Descent-lemma-stochastic} where the coefficient formula in \Cref{asmp:Coefficient-formula-stochastic} was derived).
\end{enumerate}

\end{rem}

Which inequalities to take and which residual terms to create is a matter of trial and error. Some choices lead to dead ends, whereas others lead to a viable proof.

\subsection{Bounded delay} \label{sec:Bounded-delay-stochastic}
Our main focus is on unbounded delay, because convergence under unbounded delay is a new result. It is easy, though, to modify this section's proof for the case of bounded delay, which results in a much simpler proof. Let $\epsilon_1,\ldots,\epsilon_\tau,\in(0,\infty)$ be a series of parameters, let $c\in(0,1)$, let the step size be $\eta^k=c\p{1+\sum^\tau_{l=1}\p{\frac{1}{m}\epsilon_l P_l+\frac{1}{\epsilon_l}}}^{-1}$. Then we have convergence with probability $1$. The proof uses the following Lyapunov function instead of an infinite sum version:
\begin{align*}
\xi^{k} &= \n{ x^{k} -x^*} ^{2}+\frac{1}{m}\sum_{i=1}^{\tau}c_{i}\n{ x^{k+1-i}-x^{k-i}} ^{2},\quad \text{ for } c_{i} = \sum_{l=i}^{\tau}\eps_{l}P_l.
\end{align*}

\section{Proof of Convergence for Unbounded Deterministic Delays} \label{sec:Proof-of-convergence-deterministic}
Proving convergence for deterministic delays leads to a slightly weaker convergence result. This is likely because deterministic unbounded delay is a very general condition. Below is our most general result:

\begin{thm}[Convergence  under deterministic delays]\label{thm:Convergence-deterministic-delays-full}
Consider ARock under the following conditions:

1. The block sequence $i(k)$ is a sequence of uniform IID random variables (\Cref{asmp:Block-sequence}).

2. The sequence of delay vectors $\vec{j}(0), \vec{j}(1), \vec{j}(2),\ldots$ is an arbitrary sequence in $\NN^m$, independent of $i(k)$, with $\liminf j\p{k}<\infty$ (\Cref{asmp:Deterministic-delays}).

3. Let $\eps_1, \eps_2, \ldots \in (0,\infty)$ be an arbitrary sequence of parameters such that $\sum^\infty_{l=1}\eps_l<\infty$.

4. The step size is set to $\eta^k=ch_{j(k)}$ for some arbitrary fixed $c\in(0,1)$ and $h_j = \p{1+\frac{1}{m} \n{ \eps_{i}} _{\ell^{1}} + \sum_{i=1}^{j}\frac{1}{\eps_{i}}}^{-1}$.

Then with probability $1$, the sequence of ARock iterates converges weakly to a solution on subsequences of bounded delay (\Cref{def:Convergence-subsequences-bd-delay}).

\end{thm}

This theorem is proven in \Cref{sec:sub:sub:Proof-of-theorem-deterministic}. Similar to \Cref{thm:Convergence-stochastic-delays-full}, there is a sequence of parameters $\eps_1, \eps_2, \ldots$. However in the case of deterministic delays, there is no ``best'' way to chose $\epsilon_i$'s unless stronger assumptions are made on the delays. It is impossible to optimize the parameters to uniformly ensure the maximum allowable step size, since optimizing for a current delay of $j=n$ can only come at the expense of decreasing the allowable step size for other values $m\neq n$. We set these parameters to a convenient, simple choice in \Cref{sec:Parameter-choice-deterministic} to obtain \Cref{thm:Convergence-deterministic-delays-intro} presented in the introduction.

\begin{rem}[Bounded delay]
We can obtain a bounded-delay version of \Cref{thm:Convergence-deterministic-delays-full} by truncating the metric to the first $\tau$ terms as in \Cref{sec:Bounded-delay-stochastic} and setting $\eps_{\tau+1},\eps_{\tau+2},\ldots=0$. Using the step size $\eta^k=c\p{1+ \sum_{i=1}^{j}\p{\frac{1}{m}\epsilon_l+\frac{1}{\eps_{i}}}}^{-1}$ results in convergence with probability $1$.
\end{rem}

\subsection{Building a Lyapunov function}
We build a Lyapunov function in a similar way to before. Our starting point is
the Branch Point \Cref{lem:Branch-point-lemma}. Recall that
$\cF^{k}=\sigma\p{x^{0},x^{1},\ldots,x^{k},\vec{j}\p{0},\vec{j}\p{1},\ldots,\vec{j}\p{k}}$, and let the Lyapunov function $\xi^k$ be defined as before in equation \eqref{eq:def:Lyapunov}. First though, it is necessary to make an assumption on the coefficients of the Lyapunov function. The necessity of this assumption will become clear in the proof of \Cref{lem:Descent-lemma-deterministic}.

\begin{asmp}[Coefficient formula] \label{asmp:Coefficient-formula-deterministic}
Let $\eps_1, \eps_2, \ldots \in (0,\infty)$ be an arbitrary sequence of parameters such that $\sum^\infty_{l=1}\eps_l<\infty$. The coefficients of the Lyapunov function in equation \eqref{eq:def:Lyapunov} are given by $c_i=\sum_{l=i}^\infty\eps_l$.
\end{asmp}

\subsubsection{Analysis of the metric}

\begin{lem}[Descent lemma for deterministic delays]\label{lem:Descent-lemma-deterministic}
Consider the Lyapunov function $\xi^k$ defined in \eqref{eq:def:Lyapunov}. Let \Cref{asmp:Block-sequence,asmp:Deterministic-delays,asmp:Coefficient-formula-deterministic} hold. Define
\begin{align}\label{eq:Step-size-deterministic}
H_{j}=\p{1+\frac{c_{1}}{m}+\sum_{i=1}^{j}\frac{1}{\eps_{i}}}^{-1}.
\end{align}
Then ARock yields the following inequality for step size $\eta^k$:
\begin{align}
\EE\sp{\xi^{k+1}\big|\cF^{k}} &\leq \xi^{k}-\frac{\eta^{k}}{m}\n{ Sx^{k-\vec{j}\p{k}}} ^{2}\p{1-\p{\eta^{k}/h_{j\p{k}}}}.
\end{align}
\end{lem}

\begin{proof}
Start from the Branch Point Lemma \eqref{lem:Branch-point-lemma}:
\begin{align*}
\EE\sp{\xi^{k+1}\big|\cF^{k}} &\leq \n{ x^{k}} ^{2}+\frac{1}{m}\p{\sum_{i=1}^{j\p{k}}\eps_{i}\n{ x^{k+1-i}-x^{k-i}} ^{2}+\sum_{i=1}^{\infty}c_{i+1}\n{ x^{k+1-i}-x^{k-i}} ^{2}}\\
&\quad -\frac{\eta^{k}}{m}\n{ Sx^{k-\vec{j}\p{k}}} ^{2}\p{1-\eta^{k}\p{1+\frac{c_{1}}{m}+\sum_{i=1}^{j\p{k}}\frac{1}{\eps_{i}}}}\\
&\leq \n{ x^{k}} ^{2}+\frac{1}{m}\p{\sum_{i=1}^{\infty}\p{\eps_{i}+c_{i+1}}\n{ x^{k+1-i}-x^{k-i}} ^{2}}
\quad -\frac{\eta^{k}}{m}\n{ Sx^{k-\vec{j}\p{k}}} ^{2}\p{1-\p{\eta^{k}/h_{j\p{k}}}}.
\end{align*}
First assume $\eta^{k}/h_{j(k)}\leq 1$, to eliminate the last term. Ideally we have $\EE\sp{\xi^{k+1}\big|\cF^{k}} \leq  \xi^{k}$, which can be achieved with:
\begin{align*}
\n{ x^{k}}^{2} + \frac{1}{m}\p{\sum_{i=1}^{\infty}\p{c_{i+1}+\eps_{i}}\n{ x^{k+1-i}-x^{k-i}} ^{2}} &\leq \n{ x^{k}} ^{2}+\frac{1}{m}\sum_{i=1}^{\infty}c_{i}\n{ x^{k+1-i}-x^{k-i}} ^{2}.
\end{align*}
Using a similar argument to the one used in the proof of \Cref{lem:Descent-lemma-stochastic}, we obtain the coefficient formula:
\begin{align*}
c_{i} &= \sum_{l=i}^{\infty}\eps_{l}.
\end{align*}
With this choice of coefficients,  \Cref{lem:Descent-lemma-deterministic} is proven.
\end{proof}

\subsection{Convergence proof}
Now that we have built the Lyapunov function, and obtained \Cref{lem:Descent-lemma-deterministic}, it is possible to prove convergence.

\begin{lem}\label{lem:Infinite-sum-lemma-deterministic}
Consider the Lyapunov function $\xi^k$ defined in \eqref{eq:def:Lyapunov}. Let \Cref{asmp:Block-sequence,asmp:Deterministic-delays,asmp:Coefficient-formula-deterministic} hold. Define $h_j$ via equation \eqref{eq:Step-size-deterministic}. Let the step size $\eta^{k}=ch_{j\p{k}}$ for an arbitrary fixed $c\in\p{0,1}$.
Then with probability $1$, $\xi^{k}$ converges, and we have:
\begin{align}
\sum_{k=1}^{\infty}h_{j\p{k}}\n{ Sx^{k-\vec{j}\p{k}}} ^{2} &< \infty,\label{eq:lem:Sum-HSxhat-finite}\\
\sum_{k=1}^{\infty}\n{ x^{k+1}-x^{k}} ^{2} &< \infty.\label{eq:lem:Sum-differences-finite}
\end{align}
Hence $h_{j\p{k}}\n{ Sx^{k-\vec{j}\p{k}}} ^{2}\to0$
and $\n{ x^{k+1}-x^{k}} \to0$.

\end{lem}

\begin{proof}
Now $\n{x^{k+1}-x^{k}} \leq ch_{j\p{k}}\n{ Sx^{k-\vec{j}\p{k}}}$ (see \Cref{def:ARock}), and $h_{j\p{k}}\leq1$. Hence:
\begin{align*}
\sum_{k=1}^{\infty}\n{ x^{k+1}-x^{k}} ^{2} &\leq \sum_{k=1}^{\infty}c^{2}h_{j\p{k}}^{2}\n{ Sx^{k-\vec{j}\p{k}}} ^{2}
\leq \sum_{k=1}^{\infty}h_{j\p{k}}\n{ Sx^{k-\vec{j}\p{k}}} ^{2}.
\end{align*}
Clearly then, equation \eqref{eq:lem:Sum-HSxhat-finite} will imply all parts of this lemma (since any summable sequence converges to $0$).

Use the Supermartingale Convergence Theorem (\Cref{thm:Supermartingale-convergence-theorem}) on \Cref{lem:Descent-lemma-deterministic} with $\alpha^{k}=\xi^{k}$, $\gamma^k=0$, and $\theta^{k} = \frac{\eta^{k}}{m}\n{ Sx^{k-\vec{j}\p{k}}} ^{2}\p{1-\p{\eta^{k}/h_{j\p{k}}}}$. This implies that $\xi^{k}$ converges with probability $1$, and we have:
\begin{align*}
\sum_{k=1}^{\infty}\frac{ch_{j\p{k}}}{m}\n{ Sx^{k-\vec{j}\p{k}}} ^{2}\p{1-c} &< \infty,\\
\implies\sum_{k=1}^{\infty}h_{j\p{k}}\n{ Sx^{k-\vec{j}\p{k}}} ^{2} &< \infty.
\end{align*}
This proves the lemma.
\end{proof}

\subsubsection{Norm convergence} \label{sec:Norm-convergence-deterministic}

\begin{lem}\label{lem:Norm-convergence-deterministic}
Assume the conditions of  \Cref{lem:Infinite-sum-lemma-deterministic}. Then with probability $1$, $\n{ x^{k}-x^{*}} $ converges
for all $x^{*}\in\normalfont{\text{Fix}}\p{T}$.
\end{lem}

\begin{proof}
\textbf{1) Difference sum converges to $0$:}
\begin{align*}
&\quad\frac{1}{m}\sum_{i=1}^{\infty}c_{i}\n{ x^{k+1-i}-x^{k-i}} ^{2}\\
&= \p{\p{0,\ldots,0,c_1,c_2,\ldots}*\p{\ldots,\n{x^{(i-1)+1}-x^{i-1}}^2,\n{x^{i+1}-x^{i}}^2,\n{x^{(i+1)+1}-x^{(i+1)}}^2,\ldots}}(k)
\end{align*}
Hence the difference sum is the convolution of a bounded sequence that converges to $0$ as $i\to\infty$ (by  \Cref{asmp:Coefficient-formula-deterministic}), and an $\ell^1$ sequence (by  \Cref{lem:Infinite-sum-lemma-deterministic}), respectively. Notice the reversal of roles from  \Cref{prop:Norm-convergence-stochastic}. Therefore, by  \Cref{lem:Convolution-lemma}, the difference sum converges to $0$ with probability $1$.

\textbf{2) Norm Convergence:} Therefore for any particular $x^{*}\in \text{Fix}(T)$, with probability $1$, $\n{ x^{k}-x^{*}} $ converges. As argued before in the proof of \Cref{prop:Norm-convergence-stochastic}, because the space is \textit{separable}, this implies that with probability $1$, $\n{ x^{k}-x^{*}}$ converges for \textbf{all} $x^{*}\in\text{Fix}\p{T}$.
\end{proof}

\subsubsection{Fixed-point-residual strong convergence on subsequences of bounded delay}

\begin{lem}[FPR strong convergence] \label{lem:FPR-strong-convergence-deterministic}
Let the conditions of \Cref{lem:Infinite-sum-lemma-deterministic} hold. Let $J\geq\liminf\,j\p{k}$. Let $Q_{J}\subset\NN$ be
the subsequence of indices, $k$, on which the current delay, $j(k)$, is less than or equal to $J$ (see \Cref{def:Convergence-subsequences-bd-delay}). On this subsequence, we have $\n{ Sx^{k}} \to0$.

\end{lem}

\begin{proof}
\textbf{1) Delayed fixed-point residual} $\n{ Sx^{k-\vec{j}\p{k}}} \to0$
\textbf{on ${Q_{J}}$.} The starting point is \eqref{eq:lem:Sum-HSxhat-finite}
from \Cref{lem:Infinite-sum-lemma-deterministic}:
\begin{align*}
\sum_{k=1}^{\infty}h_{j\p{k}}\n{ Sx^{k-\vec{j}\p{k}}} ^{2} & < \infty,
\end{align*}
Consider the subsequence $Q_{J}\subset\NN$. On this subsequence, the above becomes:
\begin{align*}
\infty & > \sum_{k\in Q_{J}}h_{j\p{k}}\n{ Sx^{k-\vec{j}\p{k}}} ^{2}
 \geq \sum_{k\in Q_{J}}h_{T}\n{ Sx^{k-\vec{j}\p{k}}} ^{2}\qquad\text{ (since }h_{j}\text{ is decreasing in $j$)}.
\end{align*}
Hence $\infty  > \sum_{k\in Q_{J}}\n{ Sx^{k-\vec{j}\p{k}}} ^{2}$. So $\n{ Sx^{k-\vec{j}\p{k}}} \to0$ on $Q_{J}$.

\textbf{2) Fixed-point residual strong convergence.}
\begin{align*}
\n{ Sx^{k}}  & \leq \n{ Sx^{k}-Sx^{k-\vec{j}\p{k}}} +\n{ Sx^{k-\vec{j}\p{k}}} \\
& \leq 2\n{ x^{k}-x^{k-\vec{j}\p{k}}} +\n{ Sx^{k-\vec{j}\p{k}}} \\
& \leq 2\sum_{l=1}^m \n{ x_l^{k}-x_l^{k-j(k,l)}} +\n{ Sx^{k-\vec{j}\p{k}}} \\
& \leq 2\sum_{l=1}^m\sum_{i=1}^{j(k,l)} \n{ x_l^{k+1-i}-x_l^{k-i}} +\n{ Sx^{k-\vec{j}\p{k}}} \\
& \leq 2m\p{\n{ x^{k}-x^{k-1}} +\ldots+\n{ x^{k-\p{T+1}}-x^{k-T}} }+\n{ Sx^{k-j\p{k}}} \to0.
\end{align*}
The last line converges to $0$ because $\n{ x^{k}-x^{k-1}} \to0$ and $\n{ Sx^{k-j\p{k}}} \to0$. Hence $\n{ Sx^{k}} \to0$ on $Q_J$.
\end{proof}

\subsubsection{Proof of \Cref{thm:Convergence-deterministic-delays-full}} \label{sec:sub:sub:Proof-of-theorem-deterministic}
%
\begin{proof}
Norm convergence was proven in  \Cref{lem:Norm-convergence-deterministic}. FPR strong convergence on subsequences of bounded delay was proven in \Cref{lem:FPR-strong-convergence-deterministic}. Having satisfied the conditions of \Cref{prop:Convergence-of-fixed-point-algorithms}, we conclude that the sequence of ARock iterates converges to a solution with probability $1$ on subsequence of bounded delay.
\end{proof}
%
\subsection{Parameter choice} \label{sec:Parameter-choice-deterministic}
The parameters $\eps_1, \eps_2,\ldots$ are arbitrary. However, for the purposes of simplicity and demonstration, $\epsilon_l$ was set to $l^{1+\gamma}\sqrt{m}$ for $\gamma>0$ to obtain  \Cref{thm:Convergence-deterministic-delays-intro} in the introduction, from the more general  \Cref{thm:Convergence-deterministic-delays-full}. Integration is used to simplify the summations involved in obtaining the step size formula.
%
%
\pdfbookmark[0]{References}{References}
\printbibliography
%
\end{document}